\documentclass[11pt,noinfoline]{article}
\RequirePackage{color}
\RequirePackage{times}
\RequirePackage[OT1]{fontenc}
\RequirePackage[aap,amsthm,amsmath,natbib]{imsart}
\RequirePackage{amssymb,mathrsfs,eucal}
\RequirePackage[dvips,colorlinks,hypertex]{hyperref}
\RequirePackage{hypernat}

\startlocaldefs
\def\@journal{~}
\def\journal@url{~}

\newtheorem{thm}{Theorem}[section]
\newtheorem{cor}[thm]{Corollary}
\newtheorem{lem}[thm]{Lemma}
\newtheorem{prop}[thm]{Proposition}
\theoremstyle{definition}
\newtheorem{defn}[thm]{Definition}
\theoremstyle{remark}
\newtheorem{rem}[thm]{Remark}

\numberwithin{equation}{section}
\endlocaldefs

\newcommand{\A}{\mathsf{A}}
\newcommand{\argmax}{\mathop{\mathrm{argmax}}}

\newcommand{\pen}{\mathrm{pen}}
\newcommand{\one}{\boldsymbol{1}}
\newcommand{\vol}{\mathop{\mathrm{vol}}}

\begin{document}

\begin{frontmatter}

\title{On the minimal penalty for Markov order estimation}
\runtitle{Markov order estimation}

\author{\fnms{Ramon} \snm{van Handel}\ead[label=e2]{rvan@princeton.edu}}
\address{Department of Operations Research\\
\quad and Financial Engineering\\
Princeton University\\
Princeton, NJ 08544 \\
USA \\ \printead{e2}}
\affiliation{Princeton University}
\runauthor{Ramon van Handel}

\begin{abstract}
We show that large-scale typicality of Markov sample paths implies 
that the likelihood ratio statistic satisfies a law of iterated 
logarithm uniformly to the same scale.  As a consequence, the penalized 
likelihood Markov order estimator is strongly consistent for penalties 
growing as slowly as $\log\log n$ when an upper bound is imposed on the 
order which may grow as rapidly as $\log n$.  Our method of proof, using 
techniques from empirical process theory, does not rely on the explicit 
expression for the maximum likelihood estimator in the Markov case and 
could therefore be applicable in other settings.
\end{abstract}



\end{frontmatter}

\section{Introduction}
\label{sec:intro}

For the purposes of this paper, a Markov chain is a discrete time 
stochastic process $(X_k)_{k\ge 1}$, taking values in a state space $\A$ 
of finite cardinality $|\A|<\infty$, such that the conditional law of 
$X_k$ given the past $X_1,\ldots,X_{k-1}$ depends on the most recent $r$ 
states $X_{k-r},\ldots,X_{k-1}$ only.  The smallest number $r$ for which 
this assumption is satisfied is called the \emph{order} of the Markov 
chain.  It is evident that the order of a Markov chain determines the most 
parsimonious representation of the law of the process.  Thus estimation of 
the order from observed data is a problem of practical interest, which 
moreover raises interesting mathematical questions at the intersection of 
probability, statistics and information theory.

Denote by $\mathbf{P}(x_{1:n})$ the probability of the sequence 
$x_{1:n}\in\A^n$ under the law $\mathbf{P}$, and denote by $\Theta^r$ the 
collection of all laws of Markov chains whose order is at most $r$.  As 
the parameter spaces $\Theta^r\subset \Theta^{r+1}$ are increasing, the 
naive maximum likelihood estimate of the order $\hat 
r_n=\argmax_{r}\sup_{\mathbf{P}\in\Theta^r} \mathbf{P}(x_{1:n})$ fails to 
be consistent.  Instead, we intoduce the penalized likelihood order 
estimator
$$
	\hat r_n = \argmax_{0\le r<\kappa(n)}\left\{
		\sup_{\mathbf{P}\in\Theta^r}\log\mathbf{P}(x_{1:n})
		- \pen(n,r)
	\right\},
$$
where $\pen(n,r)$ is a penalty function and $\kappa(n)$ is a cutoff 
function.  The estimator is called \emph{strongly consistent} if
$\hat r_n\to r^\star$ $\mathbf{P}^\star$-a.s.\ as $n\to\infty$
whenever the law of the observations $\mathbf{P}^\star$ is the law of a 
Markov chain whose order is $r^\star$.  We aim to understand which 
penalties and cutoffs yield a strongly consistent estimator.

Results of this type date back to Finesso \cite{Fin90}, who considers the
case where the order $r^\star$ of the Markov chain $\mathbf{P}^\star$ is 
known \emph{a priori} to be bounded above by some constant $r^\star<K$.
In this setting, Finesso shows that the penalty and cutoff
$$
	\pen(n,r) = C|\A|^r\log\log n,\qquad\qquad
	\kappa(n) = K
$$
yield a strongly consistent order estimator for a sufficiently large
constant $C$ (by \cite{CMR05}, p.\ 592, it suffices to choose $C>2|\A|$).
It can be argued from the law of iterated logarithm for martingales that a 
penalty of this form is the minimal penalty that achieves strong 
consistency, so that the result is essentially optimal (in the sense that 
the probability of underestimation of the order is minimized).  However, 
the requirement imposed by the knowledge of an \emph{a priori} upper bound 
on the order is a significant drawback and is unrealistic in many 
applications.

Order estimation in the absence of an upper bound has been investigated, 
for example, by Kieffer \cite{Kie93}.  However, the penalty used there is 
significantly larger than the minimal penalty in the case of an \emph{a 
priori} upper bound.  Kieffer's conjecture that the well known BIC penalty 
$\pen(n,r)=\frac{1}{2}|\A|^r(|\A|-1)\log n$ yields a strongly consistent 
order estimator was proved by Csisz{\'a}r and Shields \cite{CS00}.  The 
best result to date, due to Csisz{\'a}r \cite{Cs02}, shows that the 
penalty and cutoff
$$
	\pen(n,r) = c|\A|^r\log n,\qquad\qquad
	\kappa(n) = \infty
$$
yield a strongly consistent order estimator for any choice of the constant 
$c>0$.  However, this penalty is still larger than the minimal penalty 
obtained by Finesso in the case of an \emph{a priori} upper bound on the 
order.  These results raise a basic question \cite{CS00,Cs02}: is the 
$\log n$ growth of the penalty the necessary price to be paid for the lack 
of a prior upper bound on the order, or is the minimal possible penalty 
$\log\log n$ already sufficient for consistency in the absence of a prior 
upper bound?

\subsection{Results of this paper}

The purpose of this paper is twofold.  

First, we will show that a penalty of order $\log\log n$ does indeed 
suffice for consistency of the Markov order estimator, provided we impose 
a cutoff of order $\kappa(n)\sim\log n$.  Remarkably, this is precisely 
the same cutoff as is required to establish the consistency of minimum 
description length (MDL) order estimators \cite{Cs02}, of which the BIC 
penalty is an approximation.  As the $\log\log n$ penalty is much smaller 
than the BIC penalty for large $n$, this constitutes a significant 
improvement over previous results.  However, the basic question posed 
above is only partially resolved, as our results fall short of 
establishing consistency of the $\log\log n$ penalty in the absence of a 
cutoff $\kappa(n)=\infty$ as is done in \cite{CS00,Cs02} for the BIC penalty.

Second, we introduce a new approach for proving consistency of order 
estimators in the absence of a prior upper bound on the order.  The 
techniques used in previous work \cite{CS00,Cs02} rely heavily on rather 
delicate explicit computations which exploit the availability of a closed 
form expression for the maximum likelihood estimator in the Markov case. 
In contrast, our method of proof, which uses techniques from empirical 
process theory \cite{vdG00,Mas07}, is entirely different and can be 
applied much more generally. The present approach could therefore provide 
a possible starting point for extending the results of Csisz{\'a}r and 
Shields to problems where an explicit expression for the maximum 
likelihood is not available, such as the challenging problem of order 
estimation in hidden Markov models (see \cite{CMR05}, Chapter 15).

\subsection{Comparison with the approach of Csisz{\'a}r and Shields}

A direct consequence of our main result is that the penalty and cutoff
$$
	\pen(n,r) = C^\star|\A|^r\log\log n,\qquad\qquad
	\kappa(n) = \alpha^\star\log n
$$
with suitable constants $C^\star$ and $\alpha^\star$, where $\alpha^\star$ 
depends on the observation law $\mathbf{P}^\star$, yield a strongly 
consistent penalized likelihood estimator (in order to obtain a strongly 
consistent order estimator which does not require prior knowledge of 
$\mathbf{P}^\star$ it suffices to choose $\kappa(n)=o(\log n)$).
The upper bound $\kappa(n)=\alpha^\star\log n$ is inherited directly from 
the \emph{large scale typicality} property which plays a central role also 
in \cite{CS00,Cs02}.  Our main result states that if large scale typicality 
holds with an upper bound $r<\kappa(2n)$ on the order, then the likelihood 
ratio statistic satisfies a law of iterated logarithm uniformly for 
$r<\kappa(n)$ (the details are in the following section).  Strong 
consistency of the penalized likelihood order estimator then follows 
directly.

It is instructive to make a comparison with the approach of 
\cite{CS00,Cs02} for the penalty $\pen(n,r) = c|\A|^r\log n$.  The proof 
of strong consistency in this setting consists of two parts. First, 
large-scale typicality is used to prove strong consistency of the 
estimator with cutoff $\kappa(n)=\alpha^\star\log n$.  Next, a separate 
argument is employed to show that the larger orders $r\ge \alpha^\star\log n$ 
are negligible.  Our result improves the first part of the proof, as we 
show that the conclusion already holds for the smaller penalty $\pen(n,r)= 
C^\star|\A|^r\log\log n$.  However, the second part of the proof is 
missing in our setting, and it is unclear whether such a result could in 
fact be established. The resolution of this problem should effectively 
identify the minimal penalty for Markov order estimation in the absence of 
a cutoff.

Let us also note that the first part of the proof in \cite{Cs02} makes use 
of a sort of truncated law of iterated logarithm for the empirical 
transition probabilities of the Markov chain.  However, the result in 
\cite{Cs02} implies that the likelihood ratio statistic grows as $\log\log 
n$ only for orders as large as $\log\log n$, while the bound grows as 
$\log n$ for orders as large as $\log n$.  Our main result shows that such 
a bound is not the best possible, resolving in the negative a question 
posed in \cite{Cs02}, p.\ 1621.

\subsection{Organization of the paper}

In Section \ref{sec:main}, we set up the notation to be used throughout 
the paper and state our main results.  In Section \ref{sec:reduction}, we 
reduce the proof of our main result to the problem of establishing a 
suitable deviation bound.  The requisite deviation bound is proved in 
Section \ref{sec:proof}.  The proof is based on an extension of a maximal 
inequality of van de Geer \cite{vdG00}, which can be found in the 
Appendix.

\section{Main results}
\label{sec:main}

Let us fix once and for all the alphabet $\A$ of finite cardinality 
$|\A|<\infty$ and the canonical space $\Omega=\A^{\mathbb{N}}$ endowed 
with its Borel $\sigma$-field and coordinate process $(X_k)_{k\ge 1}$ 
($X_k(\omega)=\omega(k)$ for $\omega\in\Omega$). We will write 
$x_{m:n}$ for a sequence $(x_m,\ldots,x_n)\in\A^{n-m+1}$. Moreover, for 
any probability measure $\mathbf{P}$ on $\Omega$, we will write 
$\mathbf{P}(x_{m:n})$ and $\mathbf{P}(x_{m:n}|x_{r:s})$ instead of 
$\mathbf{P}(X_{m:n}=x_{m:n})$ and 
$\mathbf{P}(X_{m:n}=x_{m:n}|X_{r:s}=x_{r:s})$, respectively, whenever no 
confusion can arise.  

A Markov chain is defined by a probability measure $\mathbf{P}$ 
such that for some $r\ge 0$
$$
        \mathbf{P}(x_{1:n}) =
        \mathbf{P}(x_{1:r})
        \prod_{i=r+1}^n\mathbf{P}(x_i|x_{i-r:i-1})
        \quad\mbox{for all }n\ge r,~x_{1:n}\in\A^n.
$$
We will always presume that our Markov chains are time homogeneous:
$$
        \mathbf{P}(X_{i}=x_{r+1}|X_{i-r:i-1}=x_{1:r}) =
        \mathbf{P}(x_{r+1}|x_{1:r})
        \quad\mbox{for all }i>r,~x_{1:r+1}\in\A^{r+1}.
$$
We denote by $\Theta^r$ the set of all probability measures that satisfy 
these conditions for the given value of $r$ ($\Theta^0$ is the class of 
all i.i.d.\ processes).  Note that $\Theta^r\subset\Theta^{r+1}$ for all 
$r$.  The \emph{order} of a Markov chain $\mathbf{P}$ is the smallest 
$r\ge 0$ such that $\mathbf{P}\in\Theta^r$.

Throughout the paper we fix a distinguished Markov chain 
$\mathbf{P}^\star$ of order $r^\star$, representing the true probability 
law of an observed process.  \emph{We assume that $\mathbf{P}^\star$ is 
stationary and irreducible}.  On the basis of a sequence of observations 
$x_{1:n}$ we obtain an estimate $\hat r_n$ of the true order $r^\star$ by 
maximizing the penalized likelihood
$$
        \hat r_n = \argmax_{0\le r<\kappa(n)}\left\{
                \sup_{\mathbf{P}\in\Theta^r}\log\mathbf{P}(x_{1:n})
                - \pen(n,r)
        \right\},
$$
where $\pen(n,r)$ is a penalty function and $\kappa(n)$ is a cutoff
function.  If
$$
        \hat r_n\xrightarrow{n\to\infty}r^\star\quad
        \mathbf{P}^\star\mbox{-a.s.},
$$
the estimator is called \emph{strongly consistent}.

\begin{rem}
As discussed in \cite{CS00}, the assumption that $\mathbf{P}^\star$ is 
irreducible is necessary for the order estimation problem to be well 
posed, while stationarity of $\mathbf{P}^\star$ entails no loss of 
generality.  In particular, the latter claim follows from the fact that 
any irreducible Markov chain $\mathbf{P}$ is absolutely continuous with 
respect to a stationary Markov chain $\mathbf{P}_{\rm s}$ with the same 
transition probabilities, so that strong consistency under 
$\mathbf{P}_{\rm s}$ automatically holds under $\mathbf{P}$ also.
\end{rem}

Define for any sequence $a_{1:r}\in\A^r$ and $n\ge 1$ the random variable
$$
	N_n(a_{1:r}) = 
	\sum_{i=r+1}^n\one_{x_{i-r:i-1}=a_{1:r}},
$$
that is, $N_n(a_{1:r})$ is the number of times the sequence $a_{1:r}$ 
appears as a subsequence of $x_{1:n-1}$.  By the ergodic theorem, the 
approximation $N_n(a_{1:r})/(n-r)\approx\mathbf{P}^\star(a_{1:r})$ holds 
for large $n$.  The \emph{large scale typicality} property essentially 
requires that this approximation holds uniformly for all $a_{1:r}$ with 
$r<\rho(n)$.  As in \cite{CS00,Cs02}, this idea plays an essential role 
in the proof of our main result.

\begin{defn}
\label{defn:lst}
The process $\mathbf{P}^\star$ is said to satisfy the \emph{large-scale
typicality} property with cutoff $\rho(n)$ if there exists a constant
$\eta<1$ such that
$$
	\left|
	\frac{1}{\mathbf{P}^\star(a_{1:r})}\frac{N_n(a_{1:r})}{n-r}-1
	\right|<\eta\quad\mbox{for all }a_{1:r}\in\A^r
	\mbox{ with }\mathbf{P}^\star(a_{1:r})>0,~
	r<\rho(n)
$$
eventually as $n\to\infty$ $\mathbf{P}^\star$-a.s.
\end{defn}

We are now ready to state the main result of this paper, which can be 
viewed as a law of iterated logarithm for the likelihood ratio statistic. 
A similar result was established in \cite{Fin90}, Lemma 3.4.1 for the case 
of a fixed order $r>r^\star$.  Our key innovation is that here the result 
holds uniformly over the order $r^\star<r<\kappa(n)$, where $\kappa(2n)$ 
is a cutoff for which the large-scale typicality property holds.

\begin{thm}
\label{thm:main}
Let $\kappa(n)\le n/4$ be an increasing function, such that the process 
$\mathbf{P}^\star$ satisfies the large-scale typicality property with 
cutoff $\kappa(2n)$.  Then there is a nonrandom constant $C_0>0$ 
\emph{(}depending only on $\eta$\emph{)} such that
$$
        \sup_{r^\star<r<\kappa(n)}\frac{1}{|\A|^r}\left\{
        \sup_{\mathbf{P}\in\Theta^r}\log\mathbf{P}(x_{1:n})-
        \sup_{\mathbf{P}\in\Theta^{r^\star}}\log\mathbf{P}(x_{1:n})
        \right\}\le C_0\log\log n
$$
eventually as $n\to\infty$ $\mathbf{P}^\star$-a.s.
\end{thm}

The following sections are devoted to the proof of this result. As a 
corollary, we obtain the following conclusion for the order estimation 
problem.

\begin{cor}
\label{cor:consist}
There exist constants $C^\star$ and $\alpha^\star$, where
$\alpha^\star$ depends on $\mathbf{P}^\star$, such that any penalty and 
cutoff that satisfy eventually as $n\to\infty$
$$
	\pen(n,r) = |\A|^rf(n)\log\log n,\qquad\qquad
	\kappa(n) \le \alpha^\star\log n,
$$
where $\kappa(n)\nearrow\infty$ and the function $f(n)$ satisfies
$$
	\liminf_{n\to\infty}f(n)\ge C^\star,\qquad\quad
	\lim_{n\to\infty}\frac{f(n)\log\log n}{n}=0,
$$
yield a strongly consistent Markov order estimator.
\end{cor}

\begin{proof}
First, it is easy to see (\cite{CS00}, Proposition A.1) 
that $\mathbf{P}^\star$-a.s.
$$
        \limsup_{n\to\infty}
        \frac{1}{n}\left\{
        \sup_{\mathbf{P}\in\Theta^r}\log\mathbf{P}(x_{1:n})-
        \sup_{\mathbf{P}\in\Theta^{r^\star}}\log\mathbf{P}(x_{1:n})
        \right\} \le -C
$$
for some constant $C>0$ and all $r<r^\star$.  As $\pen(n,r)/n\to 0$
as $n\to\infty$, this implies that $\mathbf{P}^\star$-a.s.\ we have
eventually as $n\to\infty$
$$
        \sup_{\mathbf{P}\in\Theta^r}\log\mathbf{P}(x_{1:n})-\pen(n,r)
        <\sup_{\mathbf{P}\in\Theta^{r^\star}}\log\mathbf{P}(x_{1:n})-
        \pen(n,r^\star)
        \quad\forall\,r<r^\star.
$$
As $\kappa(n)\ge r^\star$ for $n$ sufficiently large, this shows that
$\liminf_{n\to\infty}\hat r_n\ge r^\star$ $\mathbf{P}^\star$-a.s.

On the other hand, it is shown in \cite{CS00,Cs02} that the large-scale 
typicality property holds with cutoff $\kappa(2n)\le\alpha^\star\log 2n$
for some constant $\alpha^\star$ which depends on $\mathbf{P}^\star$
(the constant $\eta$ in Definition \ref{defn:lst} may be fixed 
arbitrarily).  By Theorem \ref{thm:main},
$$
        \sup_{r^\star<r<\kappa(n)}\frac{1}{\pen(n,r)}\left\{
        \sup_{\mathbf{P}\in\Theta^r}\log\mathbf{P}(x_{1:n})-
        \sup_{\mathbf{P}\in\Theta^{r^\star}}\log\mathbf{P}(x_{1:n})
        \right\}\le \frac{|\A|-1}{2|\A|}
$$
eventually as $n\to\infty$ $\mathbf{P}^\star$-a.s., provided $C^\star$ is 
chosen sufficiently large.  Note that
$$
        \frac{1}{\pen(n,r)-\pen(n,r^\star)} =
        \frac{1}{\pen(n,r)}\frac{|\A|^r}{|\A|^r-|\A|^{r^\star}}
        \le \frac{1}{\pen(n,r)}\frac{|\A|}{|\A|-1}
$$
for all $r>r^\star$, so we find that $\mathbf{P}^\star$-a.s.\ we have
eventually as $n\to\infty$
$$
        \sup_{\mathbf{P}\in\Theta^r}\log\mathbf{P}(x_{1:n})
        -\pen(n,r) <
        \sup_{\mathbf{P}\in\Theta^{r^\star}}\log\mathbf{P}(x_{1:n})
        -\pen(n,r^\star)
$$
for all $r^\star<r<\kappa(n)$.  Thus
$\limsup_{n\to\infty}\hat r_n\le r^\star$ $\mathbf{P}^\star$-a.s.
\end{proof}

\begin{rem}
The proofs of large-scale typicality in \cite{CS00,Cs02} actually 
establish a slightly stronger result, where the constant $\eta$ in 
Definition \ref{defn:lst} is replaced by $n^{-\beta}$ for some $\beta>0$.
This improvement is not needed for Theorem \ref{thm:main} to hold.
\end{rem}

\begin{rem}
Theorem \ref{thm:main} states that the constant $C_0$ depends only on the 
value of $\eta$ in Definition \ref{defn:lst}.  Unfortunately, the 
constants obtained by our method of proof are expected to be far from 
optimal; one can read off a value for $C_0$ of order $10^6$ in the proof 
of Theorem \ref{thm:main}, which is likely excessively large.
\end{rem}

\begin{rem}
It is not difficult to establish that there is a constant $C$ such that
$$
        \frac{1}{n}\left\{
        \sup_{\mathbf{P}\in\Theta^r}\log\mathbf{P}(x_{1:n})-
        \sup_{\mathbf{P}\in\Theta^{r^\star}}\log\mathbf{P}(x_{1:n})
        \right\} \le C
$$
for all $n$ and $r$.  It follows that
$$
        \sup_{r>(\log|\A|)^{-1}\log n}\frac{1}{\pen(n,r)}\left\{
        \sup_{\mathbf{P}\in\Theta^r}\log\mathbf{P}(x_{1:n})-
        \sup_{\mathbf{P}\in\Theta^{r^\star}}\log\mathbf{P}(x_{1:n})
        \right\}\le \frac{|\A|-1}{2|\A|}
$$
eventually as $n\to\infty$.  In order to obtain a version of Corollary 
\ref{cor:consist} with $\kappa(n)=\infty$, the key difficulty is therefore 
to deal with orders in the range $\alpha^\star\log n\le 
r\le(\log|\A|)^{-1}\log n$.  It is an open question whether it is possible 
to close this gap.
\end{rem}

\section{Reduction to a deviation bound}
\label{sec:reduction}

The proof of Theorem \ref{thm:main} consists of two steps.  In this 
section, we will prove the result assuming that the likelihood ratio 
statistic satisfies a certain deviation bound.  The requisite deviation 
bound, which is stated in the following Proposition, will be proved in the 
next section.

\begin{prop}
\label{prop:deviation}
Define $F_n=G_n\cap G_{2n}$, where $G_n$ denotes the event
$$
	\Bigg\{
        \Bigg|
        \frac{1}{\mathbf{P}^\star(a_{1:r})}
        \frac{N_n(a_{1:r})}{n-r}-1\Bigg|\le\eta
        \mbox{ for all }a_{1:r}\in\A^r
        \mbox{ with }
        \mathbf{P}^\star(a_{1:r})>0,~
        r<\rho(n)
        \Bigg\},
$$
with $\rho(n)$ increasing and $\rho(n)\le n/2$.  Then there exist
constants $C_1,C_1',C_2>0$, which can be chosen to depend only on $\eta$,
such that
$$
        \mathbf{P}^\star\left[
        F_n\cap\max_{i=n,\ldots,2n}
        \left\{\sup_{\mathbf{P}\in\Theta^r}\log\mathbf{P}(x_{1:i})-
        \log\mathbf{P}^\star(x_{1:i}|x_{1:r})\right\}\ge \varepsilon    
        \right] \le
        C_1'e^{-\varepsilon/C_1}
$$
for all $n\ge 1$, $r^\star<r<\rho(n)$, and $\varepsilon\ge C_2|\A|^r$.
\end{prop}

Conceptually, this result can be understood as follows.  It is well known 
in classical statistics that, in ``regular'' cases, the likelihood ratio 
statistic
$$
	\sup_{\mathbf{P}\in\Theta^r}\log\mathbf{P}(x_{1:n})
	-\log\mathbf{P}^\star(x_{1:n})
$$
converges weakly as $n\to\infty$ to a $\chi^2$-distributed random 
variable.  Therefore, we expect the likelihood ratio statistic to possess 
exponential tails at least for large $n$.  Proposition 
\ref{prop:deviation} provides a precise nonasymptotic description of this 
phenomenon.

We now prove Theorem \ref{thm:main} presuming that Proposition 
\ref{prop:deviation} holds.

\begin{proof}[Proof of Theorem \ref{thm:main}]
We clearly need only consider sequences $x_{1:n}$ with 
$\mathbf{P}^\star(x_{1:n})>0$.  We begin with some straightforward 
estimates:
\begin{equation*}
\begin{split}
        &
        \sup_{r^\star<r<\kappa(n)}\frac{1}{|\A|^r}\left\{
        \sup_{\mathbf{P}\in\Theta^r}\log\mathbf{P}(x_{1:n})-
        \sup_{\mathbf{P}\in\Theta^{r^\star}}\log\mathbf{P}(x_{1:n})
        \right\} \\
        &\qquad\mbox{}\le
        \sup_{r^\star<r<\kappa(n)}\frac{1}{|\A|^r}\left\{
        \sup_{\mathbf{P}\in\Theta^r}\log\mathbf{P}(x_{1:n})-
        \log\mathbf{P}^\star(x_{1:n})
        \right\} \\
        &\qquad\mbox{}=
        \sup_{r^\star<r<\kappa(n)}\frac{1}{|\A|^r}\left\{
        \sup_{\mathbf{P}\in\Theta^r}\log\mathbf{P}(x_{1:n})-
        \log\mathbf{P}^\star(x_{1:n}|x_{1:r})-
        \log\mathbf{P}^\star(x_{1:r})
        \right\} \\
        &\qquad\mbox{}\le
        \sup_{r^\star<r<\kappa(n)}\frac{1}{|\A|^r}\left\{
        \sup_{\mathbf{P}\in\Theta^r}\log\mathbf{P}(x_{1:n})-
        \log\mathbf{P}^\star(x_{1:n}|x_{1:r})\right\}
        +C,
\end{split}
\end{equation*}
for a constant $C$ independent of $n$ and $x_{1:n}$.  Here we have used 
that for any irreducible (and time homogeneous) Markov chain 
$\mathbf{P}^\star$, there exists a constant $0<\lambda<1$ such 
that $\mathbf{P}^\star(x_{1:r})>\lambda^r$ whenever 
$\mathbf{P}^\star(x_{1:r})>0$, so that
$$
	\sup_{r>r^\star}\frac{-\log\mathbf{P}^\star(x_{1:r})}{|\A|^r}
	\le C :=
	\log(1/\lambda)\,
	\sup_{r>r^\star}\frac{r}{|\A|^r}<\infty.
$$
We conclude that it suffices to prove
$$
        \sup_{r^\star<r<\kappa(n)}\frac{1}{|\A|^r}\left\{
        \sup_{\mathbf{P}\in\Theta^r}\log\mathbf{P}(x_{1:n})-
        \log\mathbf{P}^\star(x_{1:n}|x_{1:r})\right\}\le C_0\log\log n
$$
eventually as $n\to\infty$ $\mathbf{P}^\star$-a.s.  Define for simplicity
$$
        \Delta_{i,r} = 
        \sup_{\mathbf{P}\in\Theta^r}\log\mathbf{P}(x_{1:i})-
        \log\mathbf{P}^\star(x_{1:i}|x_{1:r}).
$$
We can estimate
\begin{equation*}
\begin{split}
        &
        \mathbf{P}^\star\left[
        F_{2^{n}}\cap\max_{2^n\le i\le 2^{n+1}}
        \frac{1}{\log\log i}
        \sup_{r^\star<r<\kappa(i)}
        \frac{\Delta_{i,r}}{|\A|^r}\ge C_0\right] \\
        &\qquad\mbox{}\le
        \mathbf{P}^\star\left[
        F_{2^{n}}\cap\max_{2^n\le i\le 2^{n+1}}
        \sup_{r^\star<r<\kappa(2^{n+1})}
        \frac{\Delta_{i,r}}{|\A|^r}\ge C_0\log\log 2^n\right] \\
        &\qquad\mbox{}\le
        \sum_{r^\star<r<\kappa(2^{n+1})}
        \mathbf{P}^\star\left[
        F_{2^{n}}\cap\max_{2^n\le i\le 2^{n+1}}
        \Delta_{i,r}
        \ge C_0|\A|^r\log\log 2^n\right],
\end{split}
\end{equation*}
where we used that $\kappa(n)$ is increasing.  Now let $F_n$ be defined as 
in Proposition \ref{prop:deviation} for $\rho(n)=\kappa(2n)$.  Then there 
exist $C_1,C_1'$ such that for all $n$ sufficiently large,
$$
        \mathbf{P}^\star\left[
        F_{2^{n}}\cap\max_{2^n\le i\le 2^{n+1}}
        \Delta_{i,r}
        \ge C_0|\A|^r\log\log 2^n\right] \le
	C_1'e^{-C_0|\A|^r\log\log 2^n/C_1}
$$
for all $r^*<r<\kappa(2^{n+1})$.  Therefore
\begin{equation*}
\begin{split}
        &\mathbf{P}^\star\left[
        F_{2^{n}}\cap\max_{2^n\le i\le 2^{n+1}}
        \frac{1}{\log\log i}
        \sup_{r^\star<r<\kappa(i)}
        \frac{\Delta_{i,r}}{|\A|^r}\ge C_0\right] \\
	&\qquad\qquad\mbox{}\le
        C_1'\sum_{r^\star<r<\kappa(2^{n+1})}
	\left(e^{-C_0\log\log 2/C_1}
	n^{-C_0/C_1}\right)^{|\A|^r} \\
	&\qquad\qquad\mbox{}\le
	2C_1'e^{-C_0\log\log 2/C_1}n^{-C_0/C_1}
\end{split}
\end{equation*}
for $n$ sufficiently large.  Thus for any choice of $C_0>C_1$, we find 
that
$$
	\sum_{n=1}^\infty
        \mathbf{P}^\star\left[
        F_{2^{n}}\cap\max_{2^n\le i\le 2^{n+1}}
        \frac{1}{\log\log i}
        \sup_{r^\star<r<\kappa(i)}
        \frac{\Delta_{i,r}}{|\A|^r}\ge C_0
	\right] < \infty.
$$
By the Borel-Cantelli lemma,
$$
        F_{2^{n}}^c\cup
	\max_{2^n\le i\le 2^{n+1}}
        \frac{1}{\log\log i}
        \sup_{r^\star<r<\kappa(i)}
        \frac{\Delta_{i,r}}{|\A|^r}<C_0
	\quad\mbox{eventually as }n\to\infty\quad
	\mathbf{P}^\star\mbox{-a.s.}
$$
But by large-scale typicality with cutoff $\kappa(2n)$, we know that
$F_{2^n}$ must hold eventually as $n\to\infty$ $\mathbf{P}^\star$-a.s.
The result follows immediately.
\end{proof}

\begin{rem}
The proof of Theorem \ref{thm:main} shows that the large-scale typicality 
property is in fact only needed along an exponentially increasing 
subsequence of times $t_n=2^n$, so that the assumption of the Theorem can 
be weakened slightly.  However, the weaker assumption does not ultimately 
appear to lead to better results than the full large-scale typicality 
assumption (for example, note that the proof of large-scale typicality in
\cite{CS00} already utilizes such a subsequence).
\end{rem}

\begin{rem}
Theorem \ref{thm:main} could be improved by employing the blocking 
procedure along the subsequence $t_n=\gamma^n$ for arbitrary $\gamma>1$.
In this manner, one can establish that the result is still valid under the 
weaker assumption that the large-scale typicality property holds with 
cutoff $\kappa(\gamma n)$ for some $\gamma>1$.  However, this does not 
appear to lead to a substantially different conclusion for the order 
estimation problem.  In order to keep the notation and proofs as 
transparent as possible we have restricted our results to the case 
$\gamma=2$, but the necessary modifications for the case of arbitrary 
$\gamma>1$ are easily implemented.
\end{rem}

\section{Proof of Proposition \ref{prop:deviation}}
\label{sec:proof}

The longest part of the proof of Theorem \ref{thm:main} consists of the 
proof of Proposition \ref{prop:deviation}.  To establish this result, we 
adapt an approach using techniques from empirical process theory 
\cite{vdG00,Mas07} that was originally developed to obtain rates of 
convergence for nonparametric maximum likelihood estimators in the i.i.d.\ 
setting.  At the heart of the proof of Proposition \ref{prop:deviation} 
lies an extension of a maximal inequality for families of martingales 
under bracketing entropy conditions, due to van de Geer \cite{vdG00}, 
Theorem 8.13.  The extension of this result that is needed for our 
purposes is developed in the Appendix.

\subsection{Preliminary computations}

Any measure $\mathbf{P}\in\Theta^r$ is uniquely determined by its initial 
probability $\mathbf{P}(x_{1:r})$ and its transition probability 
$\mathbf{P}(x_{r+1}|x_{1:r})$.  It is easily seen that the measure which 
maximizes the log-likelihood $\log\mathbf{P}(x_{1:n})$ of 
$\mathbf{P}\in\Theta^r$ assigns unit probability to the observed initial 
path $x_{1:r}$.  Thus for $r>r^\star$
$$
        \sup_{\mathbf{P}\in\Theta^r}
        \log\mathbf{P}(x_{1:n})
        -\log\mathbf{P}^\star(x_{1:n}|x_{1:r}) =
        \sup_{\mathbf{P}\in\Theta^r}
        \sum_{i=r+1}^n
        \log\left(\frac{\mathbf{P}(x_i|x_{i-r:i-1})}{
        \mathbf{P}^\star(x_i|x_{i-r:i-1})}\right).
$$
The family of functions $\log(\mathbf{P}(x_i|x_{i-r:i-1})/ 
\mathbf{P}^\star(x_i|x_{i-r:i-1}))$ ($\mathbf{P}\in\Theta^r$) is 
$\mathbf{P}^\star$-a.s.\ uniformly bounded from above but not from 
below.  To avoid problems later on, we apply a standard trick.
For any $\mathbf{P}\in\Theta^r$, define
$$
        \mathbf{\tilde P}(x_i|x_{i-r:i-1}) =
        \frac{\mathbf{P}(x_i|x_{i-r:i-1})+
        \mathbf{P}^\star(x_i|x_{i-r:i-1})}{2}.
$$
Thus $\mathbf{\tilde P}$ is a Markov chain whose transition 
probabilities are an equal mixture of the transition probabilities
of $\mathbf{P}$ and $\mathbf{P}^\star$ (the initial probabilities
of $\mathbf{\tilde P}$ are irrelevant for our purposes and need not be 
defined).  By concavity of the logarithm, we find
$$
        \sup_{\mathbf{P}\in\Theta^r}
        \log\mathbf{P}(x_{1:n})
        -\log\mathbf{P}^\star(x_{1:n}|x_{1:r}) \le
        2\sup_{\mathbf{P}\in\Theta^r}
        \sum_{i=r+1}^n
        \log\left(\frac{\mathbf{\tilde P}(x_i|x_{i-r:i-1})}{
        \mathbf{P}^\star(x_i|x_{i-r:i-1})}\right).
$$
It therefore suffices to obtain a deviation bound for the right hand
side of this expression, whose summands are $\mathbf{P}^\star$-a.s.\ 
uniformly bounded above and below.

\subsection{Peeling}

The first part of the proof of Proposition \ref{prop:deviation} aims to 
reduce the problem to a deviation inequality for martingales.  To this end 
we employ a peeling device from the theory of weighted empirical 
processes.

Define the natural filtration $\mathcal{F}_n=\sigma\{X_1,\ldots,X_n\}$.
For any $\mathbf{P}\in\Theta^r$, we define
$$
        M_n^{\mathbf{P}} = 
        \sum_{i=r+1}^n
	\left\{
        \log\left(\frac{\mathbf{\tilde P}(x_i|x_{i-r:i-1})}{
        \mathbf{P}^\star(x_i|x_{i-r:i-1})}\right) 
	-
	\mathbf{E}^\star\left[\left.
        \log\left(\frac{\mathbf{\tilde P}(x_i|x_{i-r:i-1})}{
        \mathbf{P}^\star(x_i|x_{i-r:i-1})}\right) 
	\right|\mathcal{F}_{i-1}\right]
	\right\},
$$
which is a martingale (under $\mathbf{P}^\star$) by construction.
It is easily seen that
$$
	M_n^{\mathbf{P}} =
        \sum_{i=r+1}^n
        \log\left(\frac{\mathbf{\tilde P}(x_i|x_{i-r:i-1})}{
        \mathbf{P}^\star(x_i|x_{i-r:i-1})}\right) 
	+ D_n^{\mathbf{P}},
$$
where we have defined
$$
        D_n^{\mathbf{P}} = 
        -\sum_{i=r+1}^n
        \sum_{a_i\in\A}
        \mathbf{P}^\star(a_i|x_{i-r:i-1})
        \log\left(\frac{\mathbf{\tilde P}(a_i|x_{i-r:i-1})}{
        \mathbf{P}^\star(a_i|x_{i-r:i-1})}\right).
$$
We also define for any $\mathbf{P},\mathbf{P}'\in\Theta^r$
the quantity
$$
        H_n(\mathbf{P},\mathbf{P}') =
        \sum_{i=r+1}^n
        \sum_{a_i\in\A}
        \left(
        \mathbf{\tilde P}(a_i|x_{i-r:i-1})^{1/2}-
        \mathbf{\tilde P}'(a_i|x_{i-r:i-1})^{1/2}
        \right)^2. 
$$
Note that $\sqrt{H_n(\mathbf{P},\mathbf{P}')}$ defines a random distance
on $\Theta^r$.  As we will see below, the role of the set $F_n$ (and hence 
the large-scale typicality assumption) in the proof of Proposition 
\ref{prop:deviation} is that it allows us to control this random distance.

\begin{lem}
\label{lem:peeling}
For any $\varepsilon>0$, $n\ge 1$ and $r>r^\star$
\begin{multline*}
        \mathbf{P}^\star\left[
        F_n\cap\max_{i=n,\ldots,2n}
        \left\{\sup_{\mathbf{P}\in\Theta^r}\log\mathbf{P}(x_{1:i})-
        \log\mathbf{P}^\star(x_{1:i}|x_{1:r})\right\}\ge \varepsilon    
        \right]\\
	\mbox{}\le
        \sum_{k=0}^\infty
        \mathbf{P}^\star\left[
        F_n\cap
        \sup_{\mathbf{P}\in\Theta^r}
        \one_{H_n(\mathbf{P},\mathbf{P}^\star)\le 2^k\varepsilon}
        \max_{i=n,\ldots,2n}M_i^{\mathbf{P}}
        \ge 2^{k-1}\varepsilon
        \right].
\end{multline*}
\end{lem}

\begin{proof}
From the discussion above, it is clear that
\begin{equation*}
\begin{split}
        &\mathbf{P}^\star\left[
        F_n\cap\max_{i=n,\ldots,2n}
        \left\{\sup_{\mathbf{P}\in\Theta^r}\log\mathbf{P}(x_{1:i})-
        \log\mathbf{P}^\star(x_{1:i}|x_{1:r})\right\}\ge \varepsilon    
        \right] \\
        &\qquad\mbox{}\le
        \mathbf{P}^\star\left[
        F_n\cap\max_{i=n,\ldots,2n}
        \sup_{\mathbf{P}\in\Theta^r}
        \sum_{\ell=r+1}^i
        \log\left(\frac{\mathbf{\tilde P}(x_\ell|x_{\ell-r:\ell-1})}{
        \mathbf{P}^\star(x_\ell|x_{\ell-r:\ell-1})}\right)
        \ge \frac{\varepsilon}{2}
        \right] \\
        &\qquad\mbox{}=
        \mathbf{P}^\star\left[
        F_n\cap\max_{i=n,\ldots,2n}
        \sup_{\mathbf{P}\in\Theta^r}
        \left\{M_i^{\mathbf{P}}-D_i^{\mathbf{P}}\right\}
        \ge \frac{\varepsilon}{2}
        \right].
\end{split}
\end{equation*}
Now note that as $-\log x\ge 2-2\sqrt{x}$ for $x>0$, 
$$
        D_n^{\mathbf{P}}\ge
        2\sum_{i=r+1}^n
        \sum_{a_i\in\A}
        \mathbf{P}^\star(a_i|x_{i-r:i-1})
        \left(1-\frac{\mathbf{\tilde P}(a_i|x_{i-r:i-1})^{1/2}}{
        \mathbf{P}^\star(a_i|x_{i-r:i-1})^{1/2}}\right) = 
        H_n(\mathbf{P},\mathbf{P}^\star).
$$
Therefore, we can estimate
\begin{equation*}
\begin{split}
        &\mathbf{P}^\star\left[
        F_n\cap\max_{i=n,\ldots,2n}
        \left\{\sup_{\mathbf{P}\in\Theta^r}\log\mathbf{P}(x_{1:i})-
        \log\mathbf{P}^\star(x_{1:i}|x_{1:r})\right\}\ge \varepsilon    
        \right] \\
        &\qquad\mbox{}\le
        \mathbf{P}^\star\left[
        F_n\cap\max_{i=n,\ldots,2n}
        \sup_{\mathbf{P}\in\Theta^r}
        \left\{M_i^{\mathbf{P}}-H_i(\mathbf{P},\mathbf{P}^\star)\right\}
        \ge \frac{\varepsilon}{2}
        \right] \\
        &\qquad\mbox{}\le
        \mathbf{P}^\star\left[
        F_n\cap
        \sup_{\mathbf{P}\in\Theta^r}
        \left\{
        \max_{i=n,\ldots,2n}M_i^{\mathbf{P}}
        -H_n(\mathbf{P},\mathbf{P}^\star)\right\}
        \ge \frac{\varepsilon}{2}
        \right].
\end{split}
\end{equation*}
We now partition the space $\Theta^r$ into an inner ring
$\{\mathbf{P}\in\Theta^r:H_n(\mathbf{P},\mathbf{P}^\star)\le\varepsilon\}$
and a collection of concentric rings $\{\mathbf{P}\in\Theta^r:
2^{k-1}\varepsilon\le H_n(\mathbf{P},\mathbf{P}^\star)\le 
2^k\varepsilon\}$ (note that this is a random partition,
as the quantity $H_n(\mathbf{P},\mathbf{P}')$ depends on the 
observed path).  Applying the union bound gives the estimates
\begin{equation*}
\begin{split}
        &\mathbf{P}^\star\left[
        F_n\cap\max_{i=n,\ldots,2n}
        \left\{\sup_{\mathbf{P}\in\Theta^r}\log\mathbf{P}(x_{1:i})-
        \log\mathbf{P}^\star(x_{1:i}|x_{1:r})\right\}\ge \varepsilon    
        \right] \\
        &\qquad\mbox{}\le
        \mathbf{P}^\star\left[
        F_n\cap
        \sup_{\mathbf{P}\in\Theta^r}
        \left\{
        \max_{i=n,\ldots,2n}M_i^{\mathbf{P}}
        -H_n(\mathbf{P},\mathbf{P}^\star)\right\}
        \one_{H_n(\mathbf{P},\mathbf{P}^\star)\le\varepsilon}
        \ge \frac{\varepsilon}{2}
        \right] \\
        &\qquad\qquad\mbox{}+
        \sum_{k=1}^\infty
        \mathbf{P}^\star\Bigg[
        F_n\cap
        \sup_{\mathbf{P}\in\Theta^r}
        \Bigg\{
        \max_{i=n,\ldots,2n}M_i^{\mathbf{P}}
        -H_n(\mathbf{P},\mathbf{P}^\star)\Bigg\}
        \\
        &\hskip6cm \mbox{}
        \times\one_{2^{k-1}\varepsilon\le 
        H_n(\mathbf{P},\mathbf{P}^\star)\le 2^k\varepsilon}
        \ge \frac{\varepsilon}{2}
        \Bigg] \\
        &\qquad\mbox{}\le
        \sum_{k=0}^\infty
        \mathbf{P}^\star\left[
        F_n\cap
        \sup_{\mathbf{P}\in\Theta^r}
        \one_{H_n(\mathbf{P},\mathbf{P}^\star)\le 2^k\varepsilon}
        \max_{i=n,\ldots,2n}M_i^{\mathbf{P}}
        \ge 2^{k-1}\varepsilon
        \right].
\end{split}
\end{equation*}
The proof is complete.
\end{proof}

\subsection{Control of $H_n$}

Our next task is to control the quantity $H_n(\mathbf{P},\mathbf{P}')$.  
First, we show that on the event $F_n$ the quantity $H_n$ is comparable to
$$
        H(\mathbf{P},\mathbf{P}') =
        \sum_{a_{1:r+1}\in\A^{r+1}}
        \mathbf{P}^\star(a_{1:r})
        \left(
        \mathbf{\tilde P}(a_{r+1}|a_{1:r})^{1/2}-
        \mathbf{\tilde P}'(a_{r+1}|a_{1:r})^{1/2}
        \right)^2,
$$
which is a nonrandom squared distance on $\Theta^r$.

\begin{lem}
\label{lem:normcontrol}
There exist constants $C_3,C_4$ such that for any $n\ge 1$, we have
$$
	H_{2n}(\mathbf{P},\mathbf{P}') \le
	C_3\,H_{n}(\mathbf{P},\mathbf{P}')
$$
and
$$
	(n-r)\,C_4^{-1}H(\mathbf{P},\mathbf{P}')
	\le
	H_{n}(\mathbf{P},\mathbf{P}') 
	\le
	(n-r)\,C_4\,H(\mathbf{P},\mathbf{P}')
$$
for all $\mathbf{P},\mathbf{P}'\in\Theta^r$ and
$r^\star<r<\rho(n)$ on the event $F_n$.
\end{lem}

\begin{proof}
It is easily seen that for any $n\ge 1$
$$
        H_n(\mathbf{P},\mathbf{P}') =
        \sum_{a_{1:r+1}\in\A^{r+1}}
        N_n(a_{1:r})
        \left(
        \mathbf{\tilde P}(a_{r+1}|a_{1:r})^{1/2}-
        \mathbf{\tilde P}'(a_{r+1}|a_{1:r})^{1/2}
        \right)^2.      
$$
On the event $F_n$, we have by construction
$$
	(1-\eta)\,\mathbf{P}^\star(a_{1:r})
	\le 
	\frac{N_n(a_{1:r})}{n-r}
	\le (1+\eta)\,\mathbf{P}^\star(a_{1:r})
$$
and
$$
	(1-\eta)\,\mathbf{P}^\star(a_{1:r})
	\le 
	\frac{N_{2n}(a_{1:r})}{2n-r}
	\le (1+\eta)\,\mathbf{P}^\star(a_{1:r})
$$
for all $a_{1:r}\in\A^r$ and $r<\rho(n)$. Here we have used that
$\rho(n)\le\rho(2n)$ as $\rho(n)$ is presumed to be increasing.
In particular, we have
$$
	N_{2n}(a_{1:r}) \le \frac{1+\eta}{1-\eta}\,\frac{2n-r}{n-r}\,
	N_n(a_{1:r}) \le 4\,\frac{1+\eta}{1-\eta}\,N_n(a_{1:r}),
$$
where we have used that $n-r>n/2$ as $r<\rho(n)<n/2$.
The result follows directly provided we choose $C_3,C_4$ (depending only
on $\eta$) sufficiently large.
\end{proof}

Next, we control the quantity $H_n(\mathbf{P},\mathbf{P}^\star)$ in terms 
of the ``Bernstein norm'' needed in order to apply the results developed 
in the Appendix.  As in the Appendix, we define the function 
$\phi(x)=e^x-x-1$.

\begin{lem}
\label{lem:bernsteincontrol}
Define for any $\mathbf{P}\in\Theta^r$, $r>r^\star$ and $n\ge 1$
$$
        R_n^{\mathbf{P}} = 
        8\sum_{i=r+1}^n
	\mathbf{E}^\star\left[\left.
	\phi\left(
	\frac{1}{2}
	\left|
        \log\left(\frac{\mathbf{\tilde P}(x_i|x_{i-r:i-1})}{
        \mathbf{P}^\star(x_i|x_{i-r:i-1})}\right) 
	\right|
	\right)
	~
	\right|\mathcal{F}_{i-1}\right].
$$
Then $R_n^{\mathbf{P}}\le 8H_n(\mathbf{P},\mathbf{P}^\star)$ for any
$\mathbf{P}\in\Theta^r$, $r>r^\star$ and $n\ge 1$.
\end{lem}

\begin{proof}
Note that $\log(\mathbf{\tilde P}(x_i|x_{i-r:i-1})/ 
\mathbf{P}^\star(x_i|x_{i-r:i-1}))\ge -\log(2)$.  By 
\cite{vdG00}, Lemma 7.1, we have $\phi(|x|)\le (e^x-1)^2$
for any $x\ge -\log(2)/2$.  Therefore
\begin{equation*}
\begin{split}
        R_n^{\mathbf{P}} &\le
        8\sum_{i=r+1}^n
	\mathbf{E}^\star\left[\left.
	\left(\frac{\mathbf{\tilde P}(x_i|x_{i-r:i-1})^{1/2}}{
        \mathbf{P}^\star(x_i|x_{i-r:i-1})^{1/2}}-1\right)^2
	\right|\mathcal{F}_{i-1}\right] \\
	& =
        8\sum_{i=r+1}^n
	\sum_{a_i\in\A}
	\mathbf{P}^\star(a_i|x_{i-r:i-1})
	\left(\frac{\mathbf{\tilde P}(a_i|x_{i-r:i-1})^{1/2}}{
        \mathbf{P}^\star(a_i|x_{i-r:i-1})^{1/2}}-1\right)^2.
\end{split}
\end{equation*}
The result follows immediately.
\end{proof}

Together with Lemma \ref{lem:peeling}, we obtain the following.

\begin{cor}
\label{cor:onion}
Define for any $\sigma>0$ the ball
$$
	\Theta^r(\sigma) = \left\{
	\mathbf{P}\in\Theta^r: H(\mathbf{P},\mathbf{P}^\star)\le
	\sigma\right\}.
$$
Then for any $\varepsilon>0$, $n\ge 1$ and $r^\star<r<\rho(n)$
\begin{multline*}
        \mathbf{P}^\star\left[
        F_n\cap\max_{i=n,\ldots,2n}
        \left\{\sup_{\mathbf{P}\in\Theta^r}\log\mathbf{P}(x_{1:i})-
        \log\mathbf{P}^\star(x_{1:i}|x_{1:r})\right\}\ge \varepsilon    
        \right]\\
	\mbox{}\le
        \sum_{k=0}^\infty
        \mathbf{P}^\star\left[
        F_n\cap
        \sup_{\mathbf{P}\in\Theta^r(C_42^k\varepsilon/(n-r))}
        \one_{R_{2n}^\mathbf{P}\le C_32^{k+3}\varepsilon}
        \max_{i\le 2n}M_i^{\mathbf{P}}
        \ge 2^{k-1}\varepsilon
        \right].
\end{multline*}
\end{cor}

The proof is straightforward and is therefore omitted.

\subsection{Control of the bracketing entropy}

We have now reduced the proof of Proposition \ref{prop:deviation} to the 
problem of estimating the summands in Corollary \ref{cor:onion}.  We aim 
to do this by applying Proposition \ref{prop:brkgauss} in the Appendix
with $\Theta\subseteq\Theta^r$,
$$
	\xi_i^\mathbf{P}=
	\left\{\begin{array}{ll}
	\log(\mathbf{\tilde P}(x_i|x_{i-r:i-1})/
	\mathbf{P}^\star(x_i|x_{i-r:i-1})) &\quad\mbox{for }i>r,\\
	0 &\quad\mbox{for }i\le r,
	\end{array}
	\right.
$$
and $K=2$.  To this end, the main remaining difficulty is to estimate
the bracketing entropy of Definition \ref{defn:brackets}.  This is our 
next order of business.

\begin{lem}
\label{lem:entropy}
Given $c>0$, there exists $C_5>0$ depending only on $c$ such that
$$
	\log\mathcal{N}(2n,\Theta^r(\sigma),F_n,2,\delta) \le
	|\A|^{r+1}\log\left(
	\frac{C_5\sqrt{(2n-r)\,\sigma}}{\delta}
	\right)
$$
for all $n\ge 1$, $r^\star<r<\rho(n)$, $\sigma>0$ and 
$0<\delta\le c\sqrt{(2n-r)\,\sigma}$.
\end{lem}

\begin{proof}
Fix $n\ge 1$, $r^\star<r<\rho(n)$, $\sigma>0$ and
$0<\delta\le c\sqrt{(2n-r)\,\sigma}$ throughout the proof.
We begin by defining the family of functions
$$
	\mathbf{T}_\beta=
	\{p:\A^{r+1}\to\mathbb{R}_+
	~:~
	\mathbf{P}^\star(a_{1:r})^{1/2}
	p(a_{1:r+1})^{1/2}\in \beta\mathbb{Z}_+~~\forall\,
	a_{1:r+1}\in\A^{r+1}\},
$$
where $\beta>0$ is to be determined in due course.   We claim 
that for any $\mathbf{P}\in\Theta^r$, there exist 
$\lambda^{\mathbf{P}},\gamma^{\mathbf{P}}\in\mathbf{T}_\beta$ such 
that for all $a_{1:r+1}\in\A^{r+1}$ with $\mathbf{P}^\star(a_{1:r})>0$
$$
	\lambda^{\mathbf{P}}(a_{1:r+1})\le
	\mathbf{P}(a_{r+1}|a_{1:r})\le 
	\gamma^{\mathbf{P}}(a_{1:r+1})
$$ 
and 
$$
	\gamma^{\mathbf{P}}(a_{1:r+1})^{1/2}-
	\lambda^{\mathbf{P}}(a_{1:r+1})^{1/2}
	\le\frac{\beta}{
	\mathbf{P}^\star(a_{1:r})^{1/2}}.
$$ 
Indeed, this follows immediately by setting 
\begin{equation*}
\begin{split}
	\lambda^{\mathbf{P}}(a_{1:r+1}) &= \left(
	\frac{\lfloor \beta^{-1}\mathbf{P}^\star(a_{1:r})^{1/2}
	\mathbf{P}(a_{r+1}|a_{1:r})^{1/2}\rfloor}{\beta^{-1}
	\mathbf{P}^\star(a_{1:r})^{1/2}}\right)^2,\\
	\gamma^{\mathbf{P}}(a_{1:r+1})&=\left(\frac{\lceil 
	\beta^{-1}\mathbf{P}^\star(a_{1:r})^{1/2}
	\mathbf{P}(a_{r+1}|a_{1:r})^{1/2}\rceil}{\beta^{-1}
	\mathbf{P}^\star(a_{1:r})^{1/2}}
	\right)^2
\end{split}
\end{equation*}
for all $a_{1:r+1}\in\A^{r+1}$ with $\mathbf{P}^\star(a_{1:r})>0$.
Therefore $\mathbf{P}^\star$-a.s.
$$
	\Lambda_i^{\mathbf{P}} :=
	\log\left(\frac{\tilde\lambda^{\mathbf{P}}(x_i|x_{i-r:i-1})}{
	\mathbf{P}^\star(x_i|x_{i-r:i-1})}\right)
	\le\xi_i^\mathbf{P}\le
	\log\left(\frac{\tilde\gamma^{\mathbf{P}}(x_i|x_{i-r:i-1})}{
	\mathbf{P}^\star(x_i|x_{i-r:i-1})}\right)
	:=\Upsilon^{\mathbf{P}}_i
$$
for all $\mathbf{P}\in\Theta^r$, $i>r$ (we set $\Lambda_i^{\mathbf{P}}=
\Upsilon^{\mathbf{P}}_i=0$ for $i\le r$), where we have defined
$\tilde\gamma^{\mathbf{P}}(x_i|x_{i-r:i-1})=
\{\gamma^{\mathbf{P}}(x_{i-r:i})+
\mathbf{P}^\star(x_i|x_{i-r:i-1})\}/2$ and
$\tilde\lambda^{\mathbf{P}}(x_i|x_{i-r:i-1})=
\{\lambda^{\mathbf{P}}(x_{i-r:i})+
\mathbf{P}^\star(x_i|x_{i-r:i-1})\}/2$.
Moreover, we can estimate
\begin{equation*}
\begin{split}
        &8\sum_{i=1}^{2n}
        \mathbf{E}\left[\left.\phi\left(\frac{\Upsilon_i^{\mathbf{P}}-
                \Lambda_i^{\mathbf{P}}}{2}
        \right)\right|\mathcal{F}_{i-1}\right] \le 
        4\sum_{i=1}^{2n}
        \mathbf{E}\left[\left.\left(\frac{ 
	\tilde\gamma^{\mathbf{P}}(x_i|x_{i-r:i-1})^{1/2}}{
	\tilde\lambda^{\mathbf{P}}(x_i|x_{i-r:i-1})^{1/2}}
	-1\right)^2\right|\mathcal{F}_{i-1}\right] \\
	&\qquad\quad\le
        8\sum_{i=r+1}^{2n}
	\sum_{a_i\in\A}
	\left(
	\tilde\gamma^{\mathbf{P}}(a_i|x_{i-r:i-1})^{1/2}-
	\tilde\lambda^{\mathbf{P}}(a_i|x_{i-r:i-1})^{1/2}
	\right)^2 \\
	&\qquad\quad\le
        4\sum_{a_{1:r+1}\in\A^{r+1}}
	N_{2n}(a_{1:r})
	\left(
	\gamma^{\mathbf{P}}(a_{1:r+1})^{1/2}-
	\lambda^{\mathbf{P}}(a_{1:r+1})^{1/2}
	\right)^2 \\
	&\qquad\quad\le
        4\beta^2\sum_{a_{1:r+1}\in\A^{r+1}}
	\frac{N_{2n}(a_{1:r})}{\mathbf{P}^\star(a_{1:r})},
\end{split}
\end{equation*}
where we have used that $\phi(x)\le (e^x-1)^2/2$ for $x\ge 0$ and
\cite{vdG00}, Lemma 4.2.  As in the proof of Lemma \ref{lem:normcontrol}, 
we find that for any $\mathbf{P}\in\Theta^r$
$$
        8\sum_{i=1}^{2n}
        \mathbf{E}\left[\left.\phi\left(\frac{\Upsilon_i^{\mathbf{P}}-
                \Lambda_i^{\mathbf{P}}}{2}
        \right)\right|\mathcal{F}_{i-1}\right] \le
        4C_4(2n-r)|\A|^{r+1}\beta^2
$$
on the event $F_n$ (as $r<\rho(n)$ by assumption).  Therefore, if we 
choose
$$
	\beta = \frac{\delta}{\sqrt{4C_4(2n-r)|\A|^{r+1}}},
$$
then $\{(\Lambda^{\mathbf{P}}_i,\Upsilon^{\mathbf{P}}_i)_{1\le 
i\le 2n}\}_{\mathbf{P}\in\Theta^r(\sigma)}$ is a
$(2n,\Theta^r(\sigma),F_n,2,\delta)$-bracketing set.  To complete 
the proof we must estimate the cardinality of this set.

We approach this problem through a well known geometric device. We can 
represent any function from $\A^{r+1}$ to $\mathbb{R}$ as a vector in 
$\mathbb{R}^{|\A|^{r+1}}$ in the obvious fashion.  In particular, for any 
$p:\A^{r+1}\to\mathbb{R}$, denote by $\iota[p]$ the representative in 
$\mathbb{R}^{|\A|^{r+1}}$ of the function $\tilde p(a_{1:r+1})=
\mathbf{P}^\star(a_{1:r})^{1/2}p(a_{1:r+1})^{1/2}$.  Then
by \cite{vdG00}, Lemma 4.2
$$
	\iota[\Theta^r(\sigma)] \subseteq 
	B(x_0,4\sqrt{\sigma})\cap\mathbb{R}^{|\A|^{r+1}}_{++},
	\qquad\quad
	x_0=\iota[\mathbf{P}^\star(a_{r+1}|a_{1:r})],
$$
where $B(x,h)$ denotes the Euclidean ball in $\mathbb{R}^{|\A|^{r+1}}$
with center $x$ and radius $h$.  On the other hand, we clearly have
$\iota[\mathbf{T}_\beta]=(\beta\mathbb{Z}_+)^{|\A|^{r+1}}\subset
\mathbb{R}^{|\A|^{r+1}}$.  Define for any $x,x'\in\mathbb{R}^{|\A|^{r+1}}$
with $x'\succ x$ the cube $[x,x']:=\{\tilde x\in\mathbb{R}^{|\A|^{r+1}}:
x\preceq \tilde x\preceq x'\}$.  Let
$$
	\Xi_\beta :=
	\{x\in(\beta\mathbb{Z}_+)^{|\A|^{r+1}}:
	[x,x+\beta\one]\cap B(x_0,4\sqrt{\sigma})\ne
	\varnothing\},
$$
where $\one\in\mathbb{R}^{|\A|^{r+1}}$ denotes the vector all of whose 
entries are one.  Then clearly
$$
	\iota[\Theta^r(\sigma)]\subseteq
	B(x_0,4\sqrt{\sigma})\cap\mathbb{R}^{|\A|^{r+1}}_{++}
	\subseteq \bigcup_{x\in\Xi_\beta}
	[x,x+\beta\one],
$$
and, in particular, it is easily established from our previous 
computations that $\mathcal{N}(2n,\Theta^r(\sigma),F_n,2,\delta) \le
|\Xi_\beta|$.  Now suppose that $x'\in[x,x+\beta\one]$ for some
$x\in\Xi_\beta$.  Then there is an $x''\in[x,x+\beta\one]$ such that
$x''\in B(x_0,4\sqrt{\sigma})$.  In particular, we have
$\|x'-B(x_0,4\sqrt{\sigma})\|_\infty\le\beta$, and therefore
$\|x'-B(x_0,4\sqrt{\sigma})\|_2\le|\A|^{(r+1)/2}\beta$, for every
$x'\in[x,x+\beta\one]$, $x\in\Xi_\beta$.  We conclude that
$$
	\bigcup_{x\in\Xi_\beta}[x,x+\beta\one]\subseteq
	B(x_0,4\sqrt{\sigma}+|\A|^{(r+1)/2}\beta).
$$
Therefore, we can estimate
\begin{multline*}
	|\Xi_\beta|\,\beta^{|\A|^{r+1}} =
	\vol\left(
	\bigcup_{x\in\Xi_\beta}[x,x+\beta\one]
	\right) \le
	\vol\left(
	B(x_0,4\sqrt{\sigma}+|\A|^{(r+1)/2}\beta)
	\right)
	\\ \mbox{} 
	=
	(4\sqrt{\sigma}+|\A|^{(r+1)/2}\beta)^{|\A|^{r+1}}
	\vol(B(0,1)).
\end{multline*}
But from \cite{Mas07}, p.\ 249 we have the estimate
$$
	\vol(B(0,1)) \le 
	\left(\frac{\sqrt{2\pi e}}{|\A|^{(r+1)/2}}\right)^{|\A|^{r+1}}.
$$
Substituting the expression for $\beta$ and rearranging, we find that
$$
	|\Xi_\beta|
	\le
	\left(\frac{
	\{(8\sqrt{C_4}+c)\sqrt{2\pi e}\}
	\sqrt{(2n-r)\sigma}}{\delta}
	\right)^{|\A|^{r+1}},
$$
where we have used that $\delta\le c\sqrt{(2n-r)\sigma}$.
The proof is easily completed.
\end{proof}

\subsection{End of the proof}

To complete the proof of Proposition \ref{prop:deviation}, it remains to 
put together the results obtained above with Proposition 
\ref{prop:brkgauss} in the Appendix.

\begin{proof}[Proof of Proposition \ref{prop:deviation}]
In the following, we will always apply Lemma \ref{lem:entropy} 
and Proposition \ref{prop:brkgauss} with the same constants $c,c_0,c_1>0$.  
The appropriate values of these constants will be determined below.  We 
will also fix $n\ge 1$, $r^\star<r<\rho(n)$ and $\varepsilon\ge 
C_2|\A|^r$, with the constant $C_2$ to be determined.

To apply Corollary \ref{cor:onion}, we invoke
Proposition \ref{prop:brkgauss} with $K=2$, $\alpha = 
2^{k-1}\varepsilon$, and $R = C_32^{k+3}\varepsilon$ (fixing $k\ge 0$ for 
the time being).  We find that
\begin{multline*}
        \mathbf{P}^\star\left[
        F_n\cap
        \sup_{\mathbf{P}\in\Theta^r(C_42^k\varepsilon/(n-r))}
        \one_{R_{2n}^\mathbf{P}\le C_32^{k+3}\varepsilon}
        \max_{i\le 2n}M_i^{\mathbf{P}}
        \ge 2^{k-1}\varepsilon
        \right] \\ \mbox{}
	\le 2\,\exp\left[
	-\frac{2^{k-5}\varepsilon}{C_3C^2(c_1+1)}
	\right],
\end{multline*}
provided that $c_0^2\ge C^2(c_1+1)$ and
$$
        c_0\int_0^{\sqrt{C_32^{k+3}\varepsilon}}
	\sqrt{\log 
	\mathcal{N}(2n,\Theta^r(\tfrac{C_42^k\varepsilon}{n-r}),F_n,2,u)}\,du
        \le 2^{k-1}\varepsilon \le c_1C_32^{k+2}\varepsilon.
$$
To ensure that the second inequality holds, it suffices to choose
$c_1 = (8C_3)^{-1}$, and the condition on $c_0$ is satisfied by choosing
$c_0=C\sqrt{(8C_3)^{-1}+1}$.  To simplify the first inequality, choose
$c=\sqrt{8C_3/C_4}$.  Then the variable $u$ in the integral satisfies
$$
	u\le\sqrt{C_32^{k+3}\varepsilon}\le 
	c\sqrt{(2n-r)C_42^k\varepsilon/(n-r)},
$$
so by Lemma \ref{lem:entropy} it suffices to ensure that
$$
	2^{k-1}\varepsilon \ge
        |\A|^{(r+1)/2}
	C\sqrt{(8C_3)^{-1}+1}
	\int_0^{\sqrt{C_32^{k+3}\varepsilon}}
	\sqrt{\log\left(
	\frac{(4C_4)^{1/2}C_5\sqrt{2^k\varepsilon}}{u}
	\right)
	}\,du,
$$
where we have used that $r<\rho(n)\le n/2$ implies $(2n-r)/(n-r)\le 4$.
Defining 
$$
	C_6 :=
	\int_0^{\sqrt{8C_3}}
	\sqrt{\log\left(
	\frac{(4C_4)^{1/2}C_5}{v}
	\right)
	}\,dv<\infty,
$$
a simple change of variables shows that the above inequality is equivalent
to
$$
	2^{k-1}\varepsilon \ge
        |\A|^{(r+1)/2}C_6C\sqrt{(8C_3)^{-1}+1}~\sqrt{2^k\varepsilon},
$$
or, equivalently,
$$
	2^{k}\varepsilon \ge
        4C_6^2C^2((8C_3)^{-1}+1)|\A|^{r+1}.
$$
But this is always satisfied if we choose 
$C_2=4C_6^2C^2((8C_3)^{-1}+1)|\A|$.

With these choices of $c,c_0,c_1,C_2$, we have thus
shown that by Corollary \ref{cor:onion}
\begin{multline*}
        \mathbf{P}^\star\left[
        F_n\cap\max_{i=n,\ldots,2n}
        \left\{\sup_{\mathbf{P}\in\Theta^r}\log\mathbf{P}(x_{1:i})-
        \log\mathbf{P}^\star(x_{1:i}|x_{1:r})\right\}\ge \varepsilon    
        \right]\\
	\mbox{}\le
        2\sum_{k=0}^\infty
	\exp\left[
	-\frac{2^{k}\varepsilon}{2^5C^2(C_3+1/8)}
	\right]
	\le
	C_1'\,
	\exp\left[
	-\frac{\varepsilon}{C_1}
	\right]
\end{multline*}
with
$$
	C_1 = 2^5C^2(C_3+1/8),\qquad\quad
	C_1' = \frac{2}{1-e^{-C_2/2^5C^2(C_3+1/8)}},
$$
where we have used $\varepsilon\ge C_2$.  This completes the proof.
\end{proof}

\appendix

\section{A maximal inequality for martingales}

The purpose of this Appendix is to obtain a deviation bound on the 
supremum of an uncountable family of martingales, extending a result of
van de Geer \cite{vdG00}.

We work on a filtered probability space 
$(\Omega,\mathcal{F},\{\mathcal{F}_i\}_{i\ge 0},\mathbf{P})$.  We are 
given a parameter set $\Theta$ and a collection $(\xi_i^\theta)_{i\ge 1}$, 
$\theta\in\Theta$ of random variables such that $\xi_i^\theta$ is 
$\mathcal{F}_i$-measurable for all $i,\theta$.  This setting will be 
presumed throughout the Appendix.  In the following we will frequently use 
the function $\phi(x)=e^x-x-1$.

\begin{defn}
\label{defn:brackets}
Let $n\in\mathbb{N}$, $F\in\mathcal{F}$, $K>0$ and $\delta>0$ be given.  A 
finite collection $\{(\Lambda_i^j,\Upsilon_i^j)_{1\le i\le 
n}\}_{j=1,\ldots,N}$ of random variables is called a 
\emph{$(n,\Theta,F,K,\delta)$-bracketing set} if 
$\Lambda_i^j,\Upsilon_i^j$ are $\mathcal{F}_i$-measurable for all $i,j$, 
and for every $\theta\in\Theta$, there is a $1\le j\le N$ (the map 
$\theta\mapsto j$ is nonrandom) such that $\mathbf{P}$-a.s.
$$
        \Lambda_i^j\le\xi_i^\theta\le\Upsilon_i^j\quad
        \mbox{for all }i=1,\ldots,n
$$
and such that
$$
        2K^2\sum_{i=1}^n
        \mathbf{E}\left[\left.\phi\left(\frac{|\Upsilon_i^j-
                \Lambda_i^j|}{K}
        \right)\right|\mathcal{F}_{i-1}\right] \le \delta^2
        \quad\mbox{on }F.
$$
We denote as $\mathcal{N}(n,\Theta,F,K,\delta)$ the cardinality $N$ of the 
smallest $(n,\Theta,F,K,\delta)$-bracketing set 
($\log\mathcal{N}(n,\Theta,F,K,\delta)$ is called the
\emph{bracketing entropy}).
\end{defn}

The following extends a result of van de Geer \cite{vdG00}, Theorem 8.13.

\begin{prop}
\label{prop:brkgauss}
Fix $K>0$, and define for all $i\ge 0$
$$
        M_i^\theta = \sum_{\ell=1}^i\{\xi_\ell^\theta-
        \mathbf{E}[\xi_\ell^\theta|\mathcal{F}_{\ell-1}]\},\qquad\quad
        R_i^\theta = 2K^2\sum_{\ell=1}^i
        \mathbf{E}\left[\left.\phi\left(\frac{|\xi_\ell^\theta|}{K}
        \right)\right|\mathcal{F}_{\ell-1}\right].
$$
There is a universal constant $C>0$ such that
for any $n\in\mathbb{N}$, $R<\infty$ and $F\in\mathcal{F}$
$$
        \mathbf{P}\left[
        F\cap
        \sup_{\theta\in\Theta}
        \one_{R_n^\theta\le R}
        \max_{i\le n}M_i^\theta \ge \alpha
        \right] 
        \le 2\,\exp\left[-\frac{\alpha^2}{C^2(c_1+1)R}\right]
$$
for any $\alpha,c_0,c_1>0$ such that $c_0^2\ge C^2(c_1+1)$ and
$$
        c_0\int_0^{\sqrt{R}}\sqrt{\log \mathcal{N}(n,\Theta,F,K,u)}\,du
        \le \alpha \le \frac{c_1R}{K}.
$$
\emph{[}For example, the choice $C=100$ works.\emph{]}
\end{prop}

\begin{rem}
Throughout, all uncountable suprema should be interpreted as essential 
suprema under the measure $\mathbf{P}$.  Thus measurability problems are 
avoided.
\end{rem}

For our purposes, the key improvement over \cite{vdG00}, Theorem 8.13 is 
that the bound in this result is given for $\max_{i\le n}M_i^\theta$ 
rather than $M_n^\theta$.  This is essential in order to employ the 
blocking procedure in the proof of Theorem \ref{thm:main}.  Rather than 
repeat the proof of \cite{vdG00}, Theorem 8.13 here with the necessary 
modifications, we take the opportunity to obtain a more general result 
from which Proposition \ref{prop:brkgauss} follows.\footnote{
	A closer look at the proof of \cite{vdG00}, Theorem 8.13
	reveals a few inconsistencies which are corrected here.
	For example, equation (A.12) in \cite{vdG00} seems to presuppose
	that $X\ge 0$ on an event $A$ implies that 
	$\mathbf{P}[X|\mathcal{G}]\ge 0$ on $A$, which need not be the
	case.  The bracketing condition given in \cite{vdG00}, Definition
	8.1 therefore seems too weak to give the desired result.
	Similarly, the version of Bernstein's inequality given as 
	\cite{vdG00}, Lemma 8.9 does not appear to be the one used
	in the proof of Theorem 8.13.}

\begin{thm}
\label{thm:chaining}
Fix $K>0$, and define for all $i\ge 0$
$$
        M_i^\theta = \sum_{\ell=1}^i\{\xi_\ell^\theta-
        \mathbf{E}[\xi_\ell^\theta|\mathcal{F}_{\ell-1}]\},\qquad\quad
        R_i^\theta = 2K^2\sum_{\ell=1}^i
        \mathbf{E}\left[\left.\phi\left(\frac{|\xi_\ell^\theta|}{K}
        \right)\right|\mathcal{F}_{\ell-1}\right].
$$
Then we have for any
$n\in\mathbb{N}$, $R<\infty$, $F\in\mathcal{F}$ and $x>0$
$$
        \mathbf{P}\left[
        F\cap
        \sup_{\theta\in\Theta}
        \one_{R_n^\theta\le R}
        \max_{i\le n}M_i^\theta \ge 
        16\,\mathcal{H}
        + 32\sqrt{Rx}
        + 16Kx
        \right] \le 2\,e^{-x},
$$
where we have written
$$
        \mathcal{H} =
        K\log\mathcal{N}(n,\Theta,F,K,\sqrt{R})
        + 4\int_0^{\sqrt{R}}\sqrt{\log \mathcal{N}(n,\Theta,F,K,u)}\,du.
$$
\end{thm}

Before we proceed, let us prove Proposition \ref{prop:brkgauss} using
Theorem \ref{thm:chaining}.

\begin{proof}[Proof of Proposition \ref{prop:brkgauss}]
Let $\alpha=\sqrt{C^2(c_1+1)R\,x}$ and assume that the given bounds on 
$\alpha$ hold.  Then we can estimate
$$
        x = \frac{\alpha^2}{C^2(c_1+1)R} \le
        \frac{c_1R}{K}\times\frac{\alpha}{C^2(c_1+1)R} \le
        \frac{\alpha}{C^2K},\qquad
	\alpha = (\sqrt{\alpha})^2
        \le \sqrt{\frac{c_1R\alpha}{K}}.
$$
On the other hand, as $\mathcal{N}(n,\Theta,F,K,\delta)$ is nonincreasing,
we have
$$
        c_0\sqrt{R\log \mathcal{N}(n,\Theta,F,K,\sqrt{R})} \le
        c_0\int_0^{\sqrt{R}}\sqrt{\log \mathcal{N}(n,\Theta,F,K,u)}\,du
        \le \alpha.
$$
Applying Theorem \ref{thm:chaining}, we find that
\begin{multline*}
        \mathbf{P}\left[
        F\cap
        \sup_{\theta\in\Theta}
        \one_{R_n^\theta\le R}
        \max_{i\le n}M_i^\theta \ge \left\{
        \frac{16c_1}{c_0^2}
        + \frac{64}{c_0}
        + \frac{32}{\sqrt{C^2(c_1+1)}} 
        + \frac{16}{C^2}
        \right\}\alpha
        \right] \\ \mbox{}
        \le 2\,\exp\left[-\frac{\alpha^2}{C^2(c_1+1)R}\right].
\end{multline*}
But using $c_0^2\ge C^2(c_1+1)\ge C^2$, we can estimate
$$
        \frac{16c_1}{c_0^2}
        + \frac{64}{c_0}
        + \frac{32}{\sqrt{C^2(c_1+1)}} 
        + \frac{16}{C^2}
        \le
        \frac{32}{C^2}
        + \frac{96}{C}
        \le 1
$$
for $C$ sufficiently large (e.g., $C=100$).
\end{proof}

The remainder of the Appendix is devoted to the proof of Theorem 
\ref{thm:chaining}.  It should be emphasized that the approach taken here 
is entirely standard in empirical process theory: the notion of bracketing 
entropy for martingales and the proof of the requisite form of Bernstein's 
inequality follows van de Geer \cite{vdG00}, while the relatively 
transparent proof of Theorem \ref{thm:chaining} closely follows the proof 
given by Massart \cite{Mas07}, Theorem 6.8 in the i.i.d.\ setting.  The 
full proofs are given here for completeness.  Note also that we have made 
no effort to optimize the constants in the proof (the constants are 
necessarily somewhat larger than those obtained in \cite{Mas07} due to the 
presence of the additional maximum $\max_{i\le n}M_i^\theta$).

\subsection{A variant of Bernstein's inequality}

The following result is a variant of Bernstein's inequality for 
martingales. It slightly improves on \cite{vdG00}, Lemma 8.11 in that we 
do not assume that $\mathbf{E}[\xi_i|\mathcal{F}_{i-1}]=0$ for all $i$ 
(though it appears that this version is implicitly used in the proof of 
\cite{vdG00}, Theorem 8.13).

\begin{prop}
\label{prop:bernstein}
Let $(\xi_i)_{i\ge 1}$ be a sequence of random variables 
such that $\xi_i$ is $\mathcal{F}_i$-measurable for all $i$, and
define the martingale
$$
        M_j = \sum_{i=1}^j\{\xi_i -
        \mathbf{E}[\xi_i|\mathcal{F}_{i-1}]\}\quad\mbox{for all }j\ge 0.
$$
Fix $K>0$, and let $(Z_j)_{j\ge 0}$ be predictable (i.e., $Z_j$ 
is $\mathcal{F}_{j-1}$-measurable) such that
$$
        \sum_{i=1}^j\mathbf{E}\left[\left.
        |\xi_i|^m
        \right|\mathcal{F}_{i-1}\right]
        \le m!K^mZ_j\quad
        \mbox{for all }m\ge 2,~j\ge 0.
$$
Then we have for all $\alpha>0$ and $Z>0$
$$
        \mathbf{P}\left[
        M_j\ge\alpha\mbox{ and }Z_j\le Z\mbox{ for some }j
        \right]\le
        \exp\left[
        -\frac{\alpha^2}{2K(\alpha+2KZ)}
        \right].
$$
\end{prop}

\begin{proof}
Given $\lambda^{-1}>K$ we define the process $(S_j)_{j\ge 0}$ as
$S_j = e^{\lambda M_j - Z_j^\lambda}$, where $Z_j^\lambda = \sum_{i=1}^j
\mathbf{E}\left[\left.\phi(\lambda|\xi_i|)\right|
\mathcal{F}_{i-1}\right]$.
Using $1+x\le e^x$, we find
$$
        \frac{S_j}{S_{j-1}} =
        e^{\lambda\xi_j-\mathbf{E}[\lambda\xi_j|\mathcal{F}_{j-1}] -
        \mathbf{E}[\phi(\lambda|\xi_j|)|\mathcal{F}_{j-1}]}
	\le
        \frac{\{1+\phi(\lambda\xi_j)+\lambda\xi_j\}
        e^{-\mathbf{E}[\lambda\xi_j|\mathcal{F}_{j-1}]}}
        {1+\mathbf{E}[\phi(\lambda|\xi_j|)|\mathcal{F}_{j-1}]}.
$$
Now using the basic property $\phi(x)\le\phi(|x|)$ and
$1+x\le e^x$, we have
\begin{equation*}
\begin{split}
        \mathbf{E}\left[\left.\frac{S_j}{S_{j-1}}\right|\mathcal{F}_{j-1}\right]
        &\le 
        e^{-\mathbf{E}[\lambda\xi_j|\mathcal{F}_{j-1}]}
        \left\{1+
        \frac{
        \mathbf{E}[\lambda\xi_j|\mathcal{F}_{j-1}]}
        {1+\mathbf{E}[\phi(\lambda|\xi_j|)|\mathcal{F}_{j-1}]}
        \right\} 
	\\ &\le
        e^{-\mathbf{E}[\lambda\xi_j|\mathcal{F}_{j-1}]}\left\{
        1+\mathbf{E}[\lambda\xi_j|\mathcal{F}_{j-1}]
        \right\} \le 1.
\end{split}
\end{equation*}
Thus $S_j$ is a positive supermartingale.  To proceed, define the 
stopping time
$$
        \tau=\min\{j:M_j\ge\alpha\mbox{ and }Z_j\le Z\}.
$$
Then $\{M_j\ge\alpha\mbox{ and }Z_j\le Z\mbox{ for some }j\}=
\{\tau<\infty\}$.  Moreover, as $\lambda^{-1}>K$
$$
        Z_j^\lambda = 
        \sum_{\ell=2}^\infty
        \frac{\lambda^\ell}{\ell!}
        \sum_{i=1}^j
        \mathbf{E}\left[\left.|\xi_i|^\ell\right|
        \mathcal{F}_{i-1}\right]
        \le
        Z_j
        \sum_{\ell=2}^\infty
        (\lambda K)^\ell
        =
        \frac{\lambda^2K^2}{1-\lambda K}\,Z_j\quad\mbox{for all }j.
$$
Therefore $Z_\tau^\lambda\le \lambda^2K^2Z_\tau/(1-\lambda K)$, and we 
can estimate
$$
        S_\tau = e^{\lambda M_\tau - Z_\tau^\lambda} \ge 
        e^{\lambda M_\tau-\lambda^2K^2Z_\tau/(1-\lambda K)} \ge
        e^{\lambda\alpha-\lambda^2K^2Z/(1-\lambda K)}\quad
        \mbox{on }\{\tau<\infty\}.
$$
We obtain, using the supermartingale property,
$$
        \mathbf{P}[\tau<\infty] \le
        \mathbf{E}[\one_{\{\tau<\infty\}}e^{\lambda^2K^2Z/(1-\lambda K)
        -\lambda\alpha}S_\tau]
        \le e^{\lambda^2K^2Z/(1-\lambda K)-\lambda\alpha}.
$$
The proof is completed by choosing $\lambda^{-1}=K+2K^2Z/\alpha$.
\end{proof}

\begin{cor}
\label{cor:bernstein}
Let $(\xi_i)_{1\le i\le n}$ be a sequence of random variables 
such that $\xi_i$ is $\mathcal{F}_i$-measurable for all $i$,
and fix $K>0$. Define $(M_j)_{0\le j\le n}$ and
$(R_j)_{0\le j\le n}$ as
$$
        M_j = \sum_{i=1}^j\{\xi_i -
        \mathbf{E}[\xi_i|\mathcal{F}_{i-1}]\},\qquad\quad
        R_j = 2K^2\sum_{i=1}^j
        \mathbf{E}\left[\left.\phi\left(\frac{|\xi_i|}{K}\right)\right|
        \mathcal{F}_{i-1}\right].
$$
Then we have for all $\alpha>0$ and $R>0$
$$
        \mathbf{P}\left[
        \max_{j\le n}M_j\ge\alpha\mbox{ and }R_n\le R
        \right]\le
        \exp\left[
        -\frac{\alpha^2}{2(K\alpha+R)}
        \right].
$$
If in addition $\|\xi_i\|_\infty\le 3U$ for all $i$, then
for all $\alpha>0$ and $R>0$
$$
        \mathbf{P}\left[
        \max_{j\le n}M_j\ge\alpha\mbox{ and }R_n\le R
        \right]\le
        \exp\left[
        -\frac{\alpha^2}{2(U\alpha+R)}
        \right].
$$
\end{cor}

\begin{proof}
To obtain the first inequality, note that for any $m\ge 2$ and $j\ge 0$
$$
        \frac{1}{m!K^m}
        \sum_{i=1}^j\mathbf{E}\left[\left.
        |\xi_i|^m
        \right|\mathcal{F}_{i-1}\right]
        \le 
        \sum_{m=2}^\infty
        \frac{1}{m!K^m}
        \sum_{i=1}^j\mathbf{E}\left[\left.
        |\xi_i|^m
        \right|\mathcal{F}_{i-1}\right]
        =
        \frac{R_j}{2K^2}.
$$
We can therefore apply Proposition \ref{prop:bernstein} with
$Z_j=R_j/2K^2$.  For the second inequality, note that 
$\|\xi_i\|_\infty\le 3U$ implies that for all $m\ge 2$ and $j\ge 0$
$$
        \sum_{i=1}^j\mathbf{E}\left[\left.
        |\xi_i|^m
        \right|\mathcal{F}_{i-1}\right]
        \le
        (3U)^{m-2}
        \sum_{i=1}^j    
        \mathbf{E}\left[\left.
        |\xi_i|^2
        \right|\mathcal{F}_{i-1}\right]
        \le
        (3U)^{m-2}R_j \le
        \frac{m!U^mR_j}{2U^2},
$$
where we used that $m!\ge 2\times 3^{m-2}$ for $m\ge 2$.  We can therefore 
apply Proposition \ref{prop:bernstein} with $Z_j=R_j/2U^2$.   It remains 
to use that $R_j$ is nondecreasing.
\end{proof}

\subsection{Maximal inequalities for finite sets}

The following result allows us to control finite families of random 
variables that satisfy a Bernstein-type deviation inequality.  A sharper
form of this result can be obtained using an estimate on the moment 
generating function of the random variables, see \cite{Mas07}, Lemma 2.3,
but we do not have such an estimate for the maximum 
$\max_{i\le n}M_i^\theta$.  Throughout the remainder of the Appendix,
we define $\mathbf{E}^A[X]=\mathbf{E}[\one_AX]/\mathbf{P}[A]$ for any 
event $A\in\mathcal{F}$.

\begin{lem}
Let $X_1,\ldots,X_N$ be random variables such that
$$
        \mathbf{P}[|X_i|\ge\alpha] \le \exp\left[
        -\frac{\alpha^2}{2(K\alpha+R)}\right]
        \quad\mbox{for all }1\le i\le N.
$$
Then we have for any event $A\in\mathcal{F}$
$$
        \mathbf{E}^A\left[\max_{i=1,\ldots,N}|X_i|\right]
        \le \sqrt{8R\,\log\left(1+\frac{N}{\mathbf{P}[A]}\right)} +
        8K\,\log\left(1+\frac{N}{\mathbf{P}[A]}\right).
$$
\end{lem}

\begin{proof}
Let $\psi(x)$ be a Young function.  Then
\begin{multline*}
        \psi\left(\frac{
        \mathbf{E}^A\left[\max_{i\le N}|X_i|\right]}{
        \max_{i\le N}\|X_i\|_\psi}
        \right) \le
        \mathbf{E}^A\left[\max_{i\le N}
        \psi\left(\frac{|X_i|}{\|X_i\|_\psi}\right)\right]
	\\
	\mbox{}
        \le
        \sum_{i\le N}\mathbf{E}^A\left[\psi\left(
        \frac{|X_i|}{\|X_i\|_\psi}\right)\right]
        \le
	\frac{1}{\mathbf{P}[A]}
        \sum_{i\le N}\mathbf{E}\left[\psi\left(
        \frac{|X_i|}{\|X_i\|_\psi}\right)\right]
        \le
        \frac{N}{\mathbf{P}[A]},
\end{multline*}
where $\|\cdot\|_\psi$ denotes the Orlicz norm.  Therefore
$$
        \mathbf{E}^A\left[\max_{i=1,\ldots,N}|X_i|\right] \le
        \psi^{-1}\left(
        \frac{N}{\mathbf{P}[A]}
        \right)
        \max_{i=1,\ldots,N}\|X_i\|_\psi.
$$
To proceed, note that for $1\le i\le N$
\begin{equation*}
\begin{split}
        &\mathbf{P}[|X_i|\one_{|X_i|\le R/K}\ge\alpha] =
        \mathbf{P}[R/K\ge |X_i|\ge\alpha]
        \le 
        \exp\left[
        -\frac{\alpha^2}{4R}\right],\\
        &\mathbf{P}[|X_i|\one_{|X_i|\ge R/K}\ge\alpha] =
        \mathbf{P}[|X_i|\ge\alpha\vee R/K] \le
        \exp\left[-\frac{\alpha}{4K}\right].
\end{split}
\end{equation*}
By \cite{VW96}, Lemma 2.2.1, $\|X_i\one_{|X_i|\le R/K}\|_{\psi_2}
\le\sqrt{8R}$ and $\|X_i\one_{|X_i|\ge R/K}\|_{\psi_1}\le 8K$ for all $i$, 
where $\psi_p(x)=e^{x^p}-1$.  The proof is easily completed.
\end{proof}

\begin{cor}
\label{cor:unbounded}
Let $(\xi_i^h)_{1\le i\le n}$, $h=1,\ldots,N$ be random variables 
such that $\xi_i^h$ is $\mathcal{F}_i$-measurable for all $i,h$.
Fix $K>0$, and define
$$
        M_j^h = \sum_{i=1}^j\{\xi_i^h-
        \mathbf{E}[\xi_i^h|\mathcal{F}_{i-1}]\},\qquad\quad
        R_j^h = 2K^2\sum_{i=1}^j
        \mathbf{E}\left[\left.\phi\left(\frac{|\xi_i^h|}{K}
        \right)\right|\mathcal{F}_{i-1}\right].
$$
Then we have 
$$
        \mathbf{E}^A\left[
        \max_{h=1,\ldots,N}
        \one_{R_n^h\le R}
        \max_{j\le n}M_j^h
        \right]\le
        \sqrt{8R\,\log\left(1+\frac{N}{\mathbf{P}[A]}\right)}+
        8K\,\log\left(1+\frac{N}{\mathbf{P}[A]}\right)
$$
for any event $A\in\mathcal{F}$.
If in addition $\|\xi_i^h\|_\infty\le 3U$ for all $i,h$, then
$$
        \mathbf{E}^A\left[
        \max_{h=1,\ldots,N}
        \one_{R_n^h\le R}
        \max_{j\le n}M_j^h
        \right]\le
        \sqrt{8R\,\log\left(1+\frac{N}{\mathbf{P}[A]}\right)}+
        8U\,\log\left(1+\frac{N}{\mathbf{P}[A]}\right)
$$
for any event $A\in\mathcal{F}$.
\end{cor}

\begin{proof}
Apply the previous lemma with $X_h = \one_{R_n^h\le R}
\max_{j\le n}M_j^h$.  Note that as $M_0^h=0$, certainly 
$X_h\ge 0$.  Therefore $X_h = |X_h|$, and the requisite 
tail bounds are obtained immediately from Corollary 
\ref{cor:bernstein} above.
\end{proof}

\subsection{Proof of Theorem \ref{thm:chaining}}

We now proceed to the proof of Theorem \ref{thm:chaining}.  We follow 
closely the proof given by Massart \cite{Mas07}, Theorem 6.8 in the 
i.i.d.\ setting.  The general approach, by means of a chaining device 
with bracketing with adaptive truncation, is standard in empirical 
process theory.

Before we proceed to the proof, let us define the function
$$
        \Phi(x) := 16\,\mathcal{H}
        + 32\sqrt{Rx}
        + 16Kx,
$$
where $\mathcal{H}$ is as defined in Theorem \ref{thm:chaining}.
We claim that in order to prove the Theorem, it actually suffices 
to prove the estimate
$$
	\mathbf{E}^A\left[
        \sup_{\theta\in\Theta}
        \one_{R_n^\theta\le R}
        \max_{i\le n}M_i^\theta
        \right] \le 
        \Phi\left(
        \log\left(
        1+\frac{1}{\mathbf{P}[A]}
        \right)\right)
$$
for any event $A\subseteq F$.  Indeed, if this is the case, then choosing
$$
	A = F \cap \left\{
        \sup_{\theta\in\Theta}
        \one_{R_n^\theta\le R}
        \max_{i\le n}M_i^\theta \ge \Phi(x)
        \right\}
$$
allows us to estimate
$$
	\Phi(x) \le 
	\mathbf{E}^A\left[
        \sup_{\theta\in\Theta}
        \one_{R_n^\theta\le R}
        \max_{i\le n}M_i^\theta
        \right] \le 
        \Phi\left(
        \log\left(
        \frac{2}{\mathbf{P}[A]}
        \right)\right),
$$
from which the conclusion of Theorem \ref{thm:chaining} is immediate.  We 
therefore concentrate without loss of generality on obtaining the above 
estimate.

\begin{proof}[Proof of Theorem \ref{thm:chaining}]
We fix $n\in\mathbb{N}$, $K,R<\infty$, $F\in\mathcal{F}$ and $A\subseteq 
F$ throughout the proof.  Define $\delta_j=2^{-j}\sqrt{R}$ and 
$N_j=\mathcal{N}(n,\Theta,F,K,\delta_j)$ for $j\ge 0$.  We assume that 
$N_j<\infty$ for all $j$, otherwise there is nothing to prove.  Therefore, 
for each $j$, we can choose a collection
$\mathcal{B}_j=\{(\Lambda_i^{j,\rho},\Upsilon_i^{j,\rho})_{1\le i\le 
n}\}_{\rho=1,\ldots,N_j}$ that satisfies the conditions of Definition 
\ref{defn:brackets}, and these will remain fixed throughout the proof.
In particular, for every $j,\theta$, there exists $\rho(j,\theta)$ such 
that 
$$
        \Lambda_i^{j,\rho(j,\theta)}\le\xi_i^\theta\le
        \Upsilon_i^{j,\rho(j,\theta)}
        \quad\mbox{for all }i=1,\ldots,n.
$$
For notational simplicity, we will write
$$
        \Pi_i^{j,\theta}=\Upsilon_i^{j,\rho(j,\theta)},\qquad\quad
        \Delta_i^{j,\theta} = \Upsilon_i^{j,\rho(j,\theta)}-
        \Lambda_i^{j,\rho(j,\theta)}.
$$
At the heart of the proof is a chaining device: we introduce the
telescoping sum
\begin{multline*}
        \xi_i^\theta = 
        \{\xi_i^\theta - \Pi_i^{\tau_i^\theta,\theta}\wedge
        \Pi_i^{\tau_i^\theta-1,\theta}\} +
        \{\Pi_i^{\tau_i^\theta,\theta}\wedge
        \Pi_i^{\tau_i^\theta-1,\theta} - \Pi_i^{\tau_i^\theta-1,\theta}\} 
	\\ \mbox{}
	+
        \sum_{j=1}^{\tau_i^\theta-1}\{\Pi_i^{j,\theta}
        - \Pi_i^{j-1,\theta}\} +
        \Pi_i^{0,\theta},
\end{multline*}
where by convention $\Pi_i^{-1,\theta}=\Pi_i^{0,\theta}$.
The length of the chain is chosen adaptively:
$$
        \tau_i^\theta = \min\{j\ge 0:\Delta_i^{j,\theta}>a_j\}\wedge J.
$$
The levels $a_j>0$ and $J\ge 1$ will be determined later on (we will 
choose $a_j$ to control the second term in Corollary \ref{cor:unbounded}, 
and we will ultimately let $J\to\infty$).

It will be convenient to split the chain into three parts:
\begin{eqnarray}
\nonumber
        \xi_i^\theta &=& 
\label{eq:tmone}
        \Pi_i^{0,\theta} + \mbox{} \\
        &&
\label{eq:tmtwo}
        \sum_{j=0}^J
        (\xi_i^\theta - \Pi_i^{j,\theta}\wedge
        \Pi_i^{j-1,\theta})\one_{\tau_i^\theta=j} + \mbox{} \\
        &&
\label{eq:tmthree}
        \sum_{j=1}^J \left\{
        (\Pi_i^{j,\theta}\wedge\Pi_i^{j-1,\theta} - \Pi_i^{j-1,\theta})
        \one_{\tau_i^\theta=j} +
        (\Pi_i^{j,\theta}-\Pi_i^{j-1,\theta})\one_{\tau_i^\theta>j} 
        \right\}.
\end{eqnarray}
Denote by $b_i^{j,\theta}$ the summands in (\ref{eq:tmtwo})
by $c_i^{j,\theta}$ the summands in (\ref{eq:tmthree}), and
define the martingales
$A_i^\theta = \sum_{\ell=1}^i\{\Pi_\ell^{0,\theta}-
\mathbf{E}[\Pi_\ell^{0,\theta}|\mathcal{F}_{\ell-1}]\}$,
$B_i^{j,\theta} = \sum_{\ell=1}^i\{b_\ell^{j,\theta}-
\mathbf{E}[b_\ell^{j,\theta}|\mathcal{F}_{\ell-1}]\}$, and
$C_i^{j,\theta} = \sum_{\ell=1}^i\{c_\ell^{j,\theta}-
\mathbf{E}[c_\ell^{j,\theta}|\mathcal{F}_{\ell-1}]\}$.
We will control each martingale separately.

\textbf{Control of $\boldsymbol{A}^\theta$}.  As $\phi$ is convex and 
nondecreasing, and as $|\Pi_\ell^{0,\theta}-\xi_\ell^\theta|\le
|\Delta_\ell^{0,\theta}|$,
$$
        \phi\left(\frac{|\Pi_\ell^{0,\theta}|}{2K}\right) \le
        \phi\left(\frac{|\Pi_\ell^{0,\theta}-\xi_\ell^\theta|
        +|\xi_\ell^\theta|}{2K}\right) \le
        \frac{1}{2}\,
        \phi\left(\frac{|\Delta_\ell^{0,\theta}|}{K}\right) 
        +\frac{1}{2}\,
        \phi\left(\frac{|\xi_\ell^\theta|}{K}\right).
$$
Using Definition \ref{defn:brackets}, we find that
$$
        R_n^{0,\theta}:=
        8K^2\sum_{\ell=1}^n\mathbf{E}\left[\left.
        \phi\left(\frac{|\Pi_\ell^{0,\theta}|}{2K}\right) 
        \right|\mathcal{F}_{\ell-1}\right] \le
        2(\delta_0^2+R)=4R\quad
        \mbox{on }\{R_n^\theta\le R\}\cap F.
$$
Therefore
\begin{multline*}
        \mathbf{E}^A\left[
        \sup_{\theta\in\Theta}
        \one_{R_n^\theta\le R}
        \max_{i\le n}A_i^\theta
        \right] \le
        \mathbf{E}^A\left[
        \sup_{\theta\in\Theta}
        \one_{R_n^{0,\theta}\le 2(\delta_0^2+R)}
        \max_{i\le n}A_i^\theta
        \right] \\ \mbox{} \le
        \sqrt{32R\,
        \log\left(1+\frac{N_0}{\mathbf{P}[A]}\right)}+
        16K\,\log\left(1+\frac{N_0}{\mathbf{P}[A]}\right)
\end{multline*}
by Corollary \ref{cor:unbounded}, where we have used that $A\subseteq F$.

\textbf{Control of $\boldsymbol{B}^\theta$}.
Note that $b_\ell^{j,\theta}\le 0$, so that
$$
        b_\ell^{j,\theta} - 
        \mathbf{E}[b_\ell^{j,\theta}|\mathcal{F}_{\ell-1}] \le
        \mathbf{E}[
        (\Pi_\ell^{j,\theta}\wedge\Pi_\ell^{j-1,\theta} - \xi_\ell^\theta)
        \one_{\tau_\ell^\theta=j}       
        |\mathcal{F}_{\ell-1}]
        \le
        \mathbf{E}[\Delta_\ell^{j,\theta}\one_{\tau_\ell^\theta=j}      
        |\mathcal{F}_{\ell-1}].
$$
Consider first the case that $j<J$.  When $\tau_\ell^\theta=j$, we have 
$\Delta_\ell^{j,\theta}>a_j$.  Thus
$$
        b_\ell^{j,\theta} - 
        \mathbf{E}[b_\ell^{j,\theta}|\mathcal{F}_{\ell-1}] \le
        \frac{1}{a_j}\,
        \mathbf{E}[|\Delta_\ell^{j,\theta}|^2|\mathcal{F}_{\ell-1}]
        \le
        \frac{2K^2}{a_j}\,
        \mathbf{E}\left[\left.\phi\left(
        \frac{|\Delta_\ell^{j,\theta}|}{K}\right)
        \right|\mathcal{F}_{\ell-1}\right],
$$
where we have used $|x|^2\le 2K^2\phi(|x|/K)$.  In particular,
$$
        B_i^{j,\theta} \le
        \frac{2K^2}{a_j}
        \sum_{\ell=1}^i
        \mathbf{E}\left[\left.\phi\left(
        \frac{|\Delta_\ell^{j,\theta}|}{K}\right)
        \right|\mathcal{F}_{\ell-1}\right] \le
        \frac{\delta_j^2}{a_j}\quad\mbox{on }F,
$$
where we have applied Definition \ref{defn:brackets}.
As $A\subseteq F$, it follows that
$$
        \mathbf{E}^A\left[
        \sup_{\theta\in\Theta}
        \one_{R_n^\theta\le R}
        \max_{i\le n}B_i^{j,\theta}
        \right] \le \frac{\delta_j^2}{a_j}\quad
        \mbox{for }j<J.
$$
Now consider the case $j=J$.  We can estimate
$$
        B_i^{j,\theta} \le
        \sum_{\ell=1}^i
        \mathbf{E}[\Delta_\ell^{J,\theta}|\mathcal{F}_{\ell-1}]
        \le
        \left[i\sum_{\ell=1}^i
        \mathbf{E}[|\Delta_\ell^{J,\theta}|^2|\mathcal{F}_{\ell-1}]
        \right]^{1/2}\le \delta_J\sqrt{i}\quad\mbox{on }F,
$$
where we have applied the same computations as above.  It follows that
$$
        \mathbf{E}^A\left[
        \sup_{\theta\in\Theta}
        \one_{R_n^\theta\le R}
        \max_{i\le n}B_i^{J,\theta}
        \right] \le \delta_J\sqrt{n},
$$
where we have used that $A\subseteq F$.

\textbf{Control of $\boldsymbol{C}^\theta$}.
As $\Pi_\ell^{j,\theta}-\Pi_\ell^{j-1,\theta} = 
\Pi_\ell^{j,\theta}-\xi_\ell^\theta 
+\xi_\ell^\theta-\Pi_\ell^{j-1,\theta}$,
we have
$$
        -\Delta_\ell^{j-1,\theta}
        \le
        \Pi_\ell^{j,\theta}-\Pi_\ell^{j-1,\theta}
        \le
        \Delta_\ell^{j,\theta},\qquad\quad
        -\Delta_\ell^{j-1,\theta}
        \le
        \Pi_\ell^{j,\theta}\wedge\Pi_\ell^{j-1,\theta} - 
        \Pi_\ell^{j-1,\theta} \le 0.
$$
Therefore
$$
        -\Delta_\ell^{j-1,\theta}\one_{\tau_\ell^\theta\ge j}
        \le
        c_\ell^{j,\theta}
        \le 
        \Delta_\ell^{j,\theta}\one_{\tau_\ell^\theta>j}.
$$
As $\Delta_\ell^{j,\theta}\le a_j$ whenever $\tau_\ell^\theta>j$, we
find that
$$
        \|c_\ell^{j,\theta}\|_\infty \le a_{j-1}\vee a_j.
$$
Moreover, as $|c_\ell^{j,\theta}|\le\Delta_\ell^{j-1,\theta}\vee
\Delta_\ell^{j,\theta}\le \Delta_\ell^{j-1,\theta}+
\Delta_\ell^{j,\theta}$, we obtain using that $\phi$ is convex and
nondecreasing (in the same manner as above for the control of 
$\boldsymbol{A}^\theta$)
$$
        R_n^{j,\theta}:=
        8K^2\sum_{\ell=1}^n\mathbf{E}\left[\left.
        \phi\left(\frac{|c_\ell^{j,\theta}|}{2K}
        \right)\right|\mathcal{F}_{\ell-1}\right] \le
        2(\delta_{j-1}^2+\delta_j^2)
        \quad\mbox{on }F,
$$
where we have used Definition \ref{defn:brackets}.  
As $A\subseteq F$, we can therefore estimate
$$
        \mathbf{E}^A\left[
        \sup_{\theta\in\Theta}
        \one_{R_n^\theta\le R}
        \max_{i\le n}C_i^{j,\theta}
        \right] \le
        \mathbf{E}^A\left[
        \sup_{\theta\in\Theta}
        \one_{R_n^{j,\theta}\le 2(\delta_{j-1}^2+\delta_j^2)}
        \max_{i\le n}C_i^{j,\theta}
        \right].
$$
Now note that $c_\ell^{j,\theta}$ depends on $\theta$ only through the 
values of $\rho(0,\theta),\ldots,\rho(j,\theta)$.  In particular, for 
fixed $j$, the supremum of $\one_{R_n^{j,\theta}\le 
2(\delta_{j-1}^2+\delta_j^2)}\max_{i\le n}C_i^{j,\theta}$
as $\theta$ varies over $\Theta$ is in fact only the maximum over a finite 
collection of random variables, whose cardinality is bounded above by
the quantity
$$
        \textbf{N}_j:=\prod_{p=0}^jN_p.
$$ 
We therefore obtain the estimate
\begin{multline*}
        \mathbf{E}^A\left[
        \sup_{\theta\in\Theta}
        \one_{R_n^\theta\le R}
        \max_{i\le n}C_i^{j,\theta}
        \right] \\ \mbox{} \le
        \sqrt{16(\delta_{j-1}^2+\delta_j^2)\,
        \log\left(1+\frac{\mathbf{N}_j}{\mathbf{P}[A]}\right)}+
        \frac{8}{3}\,(a_{j-1}\vee a_j)\,
        \log\left(1+\frac{\mathbf{N}_j}{\mathbf{P}[A]}\right),
\end{multline*}
where we have applied Corollary \ref{cor:unbounded}.

\textbf{End of the proof}.  
Note that by construction
$$
	M_i^\theta = A_i^\theta + \sum_{j=0}^J B_i^{j,\theta} +
	\sum_{j=1}^J C_i^{j,\theta}
$$
for all $i,\theta$.  Collecting the above estimates gives
\begin{multline*}
        \mathbf{E}^A\left[
        \sup_{\theta\in\Theta}
        \one_{R_n^\theta\le R}
        \max_{i\le n}M_i^\theta
        \right] 
\\ \mbox{} 
	\le \delta_J\sqrt{n} +
        \delta_0\sqrt{32\,
        \log\left(1+\frac{N_0}{\mathbf{P}[A]}\right)}+
        16K\,\log\left(1+\frac{N_0}{\mathbf{P}[A]}\right)
        + \sum_{j=0}^{J-1}\frac{\delta_j^2}{a_j} 
\\ \mbox{}
        + \sum_{j=1}^J\left\{
                \delta_j\sqrt{80\,
        \log\left(1+\frac{\mathbf{N}_j}{\mathbf{P}[A]}\right)}+
        \frac{8}{3}\,(a_{j-1}\vee a_j)\,
        \log\left(1+\frac{\mathbf{N}_j}{\mathbf{P}[A]}\right)
        \right\}.
\end{multline*}
We aim to choose $a_j$ such that the $\log(1+\mathbf{N}_j/\mathbf{P}[A])$ 
terms disappear.  Set
$$
        a_j = \delta_j\left(
        \frac{8}{3}\,
        \log\left(1+\frac{\mathbf{N}_{j+1}}{\mathbf{P}[A]}\right)
        \right)^{-1/2}.
$$
Then $a_j$ is decreasing with increasing $j$, so $a_{j-1}\vee a_j=a_{j-1}$ 
and
\begin{multline*}
        \mathbf{E}^A\left[
        \sup_{\theta\in\Theta}
        \one_{R_n^\theta\le R}
        \max_{i\le n}M_i^\theta
        \right] 
	\\ \mbox{}
	\le \delta_J\sqrt{n} +
        16K\,\log\left(1+\frac{N_0}{\mathbf{P}[A]}\right)
        + 16\sum_{j=0}^J\delta_j
        \sqrt{\log\left(1+\frac{\mathbf{N}_j}{\mathbf{P}[A]}\right)}.
\end{multline*}
We now estimate as follows:
$$
        \sum_{j=0}^J\delta_j
        \sqrt{\log\left(1+\frac{\mathbf{N}_j}{\mathbf{P}[A]}\right)} \le
        \sum_{j=0}^J\delta_j
        \sqrt{\log\left(1+\frac{1}{\mathbf{P}[A]}\right)} + 
        \sum_{j=0}^J\delta_j\sum_{p=0}^j\sqrt{\log N_p},
$$
and
\begin{multline*}
        \sum_{j=0}^J\delta_j\sum_{p=0}^j\sqrt{\log N_p}
        \le     
        \sum_{p=0}^\infty\sqrt{\log N_p}
        \sum_{j=0}^J\delta_j\one_{p\le j}
        \le
        \sum_{p=0}^\infty\sqrt{\log N_p}
        \sum_{j=p}^\infty\delta_j
        = \mbox{} \\
        4\sum_{p=0}^\infty(\delta_p-\delta_{p+1})\,\sqrt{\log N_p} \le
        4\int_0^{\sqrt{R}}\sqrt{\log \mathcal{N}(n,\Theta,F,K,u)}\,du.
\end{multline*}
We obtain
$$
        \mathbf{E}^A\left[
        \sup_{\theta\in\Theta}
        \one_{R_n^\theta\le R}
        \max_{i\le n}M_i^\theta
        \right] \le 
        \delta_J\sqrt{n}
	+\Phi\left( 
        \log\left(1+\frac{1}{\mathbf{P}[A]}\right)\right).
$$
The result follows by letting $J\to\infty$.
\end{proof}

\bibliographystyle{acmtrans-ims}
\bibliography{ref}

\end{document}